\documentclass{amsart}
\usepackage{amssymb}
\usepackage{mathrsfs}
\usepackage{stmaryrd}
\usepackage[all]{xy}

\newcommand{\uline}[1]{\underline{#1}}
\newcommand{\uuline}[1]{\underline{\underline{#1}}}

\newcommand{\C}{\mathbb{C}}
\newcommand{\Ql}{\mathbb{Q}_\ell}

\newcommand{\Qlb}{\overline{\mathbb{Q}}_\ell}
\newcommand{\Z}{{\mathbb{Z}}}

\newcommand{\K}{\mathbb{K}}
\newcommand{\F}{\mathbb{F}}
\renewcommand{\O}{\mathbb{O}}
\newcommand{\E}{\mathbb{E}}
\newcommand{\bk}{\Bbbk}

\newcommand{\cA}{\mathcal{A}}

\newcommand{\gr}{{\mathrm{gr}}}
\newcommand{\Lgr}{L^\gr}
\newcommand{\dgr}{\Delta^\gr}
\newcommand{\wDelta}{\widetilde{\Delta}}
\newcommand{\ngr}{\nabla^\gr}
\newcommand{\wnabla}{\widetilde{\nabla}}
\newcommand{\Gr}{\mathrm{Gr}}

\newcommand{\cT}{\mathcal{T}}
\newcommand{\Db}{D^{\mathrm{b}}}
\newcommand{\Kb}{K^{\mathrm{b}}}

\newcommand{\sT}{\mathsf{T}}

\newcommand{\bleq}{\trianglelefteq}
\newcommand{\bgeq}{\trianglerighteq}

\newcommand{\cF}{\mathcal{F}}
\newcommand{\cG}{\mathcal{G}}

\newcommand{\scS}{\mathscr{S}}
\newcommand{\scT}{\mathscr{T}}
\newcommand{\Parity}{\mathsf{Parity}}

\newcommand{\mix}{\mathrm{mix}}
\newcommand{\Dmix}{D^\mix}
\newcommand{\Perv}{\mathsf{P}}
\newcommand{\dmix}{\Delta^{\mix}}
\newcommand{\nmix}{\nabla^{\mix}}
\newcommand{\uuE}{\uuline{\E}{}}

\newcommand{\uuF}{\uuline{\F}{}}
\newcommand{\cE}{\mathcal{E}}
\newcommand{\D}{\mathbb{D}}
\newcommand{\p}{{}^p\!}
\newcommand{\pH}{\p\mathcal{H}}

\newcommand{\oo}{{}^\circ\!}

\newcommand{\IC}{\mathcal{IC}}
\newcommand{\cP}{\mathcal{P}}

\newcommand{\Radon}{\mathsf{R}}

\newcommand{\cB}{\mathscr{B}}

\newcommand{\Gv}{\check{G}}
\newcommand{\Bv}{\check{B}}
\newcommand{\Tv}{\check{T}}
\newcommand{\cBv}{\check{\cB}}

\newcommand{\ICv}{{\check \IC}}
\newcommand{\dv}{\check{\Delta}}
\newcommand{\nv}{\check{\nabla}}
\newcommand{\cTv}{\check{\cT}}
\newcommand{\cEv}{\check{\cE}}
\newcommand{\muv}{{\check \mu}}

\newcommand{\iv}{{\check \imath}}

\newcommand{\simto}{\xrightarrow{\sim}}
\newcommand{\la}{\langle}
\newcommand{\ra}{\rangle}

\DeclareMathOperator{\Hom}{Hom}

\DeclareMathOperator{\Ext}{Ext}

\DeclareMathOperator{\Irr}{Irr}
\newcommand{\id}{\mathrm{id}}

\makeatletter
\def\lotimes{\@ifnextchar_{\@lotimessub}{\@lotimesnosub}}
\def\@lotimessub_#1{\mathchoice{\mathbin{\mathop{\otimes}^L}_{#1}}%
  {\otimes^L_{#1}}{\otimes^L_{#1}}{\otimes^L_{#1}}}
\def\@lotimesnosub{\mathbin{\mathop{\otimes}^L}}
\makeatother

\newtheorem*{thm*}{Theorem}
\numberwithin{equation}{section}
\newtheorem{thm}{Theorem}[section]
\newtheorem{lem}[thm]{Lemma}
\newtheorem{prop}[thm]{Proposition}
\newtheorem{cor}[thm]{Corollary}
\theoremstyle{definition}
\newtheorem{defn}[thm]{Definition}
\theoremstyle{remark}
\newtheorem{rmk}[thm]{Remark}
\newtheorem{ex}[thm]{Example}

\title[Modular perverse sheaves on flag varieties III]{Modular perverse sheaves on flag varieties III:\\ Positivity conditions}

\author{Pramod N. Achar}
\address{Department of Mathematics\\
  Louisiana State University\\
  Baton Rouge, LA 70803\\
  U.S.A.}
\email{pramod@math.lsu.edu}

\author{Simon Riche}
\address{Universit{\'e} Blaise Pascal - Clermont-Ferrand II, Laboratoire de Math{\'e}matiques, CNRS, UMR 6620, Campus universitaire des C{\'e}zeaux, F-63177 Aubi{\`e}re Cedex, France
}
\email{simon.riche@math.univ-bpclermont.fr}

\thanks{P.A. was supported by NSF Grant No.~DMS-1001594.  S.R. was supported by ANR Grants No.~ANR-09-JCJC-0102-01, ANR-2010-BLAN-110-02 and ANR-13-BS01-0001-01.}

\begin{document}

\begin{abstract}
We further develop the general theory of the ``mixed modular derived category'' introduced by the authors in a previous paper in this series. We then use it to study positivity and $Q$-Koszulity phenomena on flag varieties.
\end{abstract}

\maketitle

\section{Introduction}
\label{sec:intro}

\subsection{}
\label{ss:intro-intro}

\newcommand{\stalkC}{(1)$_\C$}
\newcommand{\kosC}{(2)$_\C$}
\newcommand{\stalkF}{(1)$_\F$}
\newcommand{\kosF}{(2)$_\F$}

The category $\Perv_{(B)}(\cB,\C)$ of Bruhat-constructible perverse $\C$-sheaves on the flag variety $\cB$ of a complex connected reductive algebraic group $G$ has been extensively studied for decades, with much of the motivation coming from applications to the representation theory of complex semisimple Lie algebras.  Two salient features of this category are as follows :
\begin{enumerate}
\item[\stalkC] 
The stalks and costalks of the simple perverse sheaves $\IC_w(\C)$ enjoy a parity-vanishing property (see \cite{kl}).
\item[\kosC]
The category $\Perv_{(B)}(\cB,\C)$ admits a Koszul grading (see \cite{bgs}).
\end{enumerate}
It was long expected that the obvious analogues of statements~\stalkC{} and~\kosC{} would also hold for modular perverse sheaves (i.e.~for perverse sheaves with coefficients in a finite field $\F$ of characteristic $\ell>0$) under mild restrictions on $\ell$, with consequences for the representation theory of algebraic groups; see e.g.~\cite{soergel}.  But Williamson's work~\cite{williamson} implies that both of these statements fail in a large class of examples. 

The next question one may want to consider is then:
what could take the place of~\stalkC{} and~\kosC{} in the setting of modular perverse sheaves?  Fix a finite extension $\K$ of $\Ql$ whose ring of integers $\O$ has $\F$ as residue field.  In this paper, we
consider the following statements as possible substitutes for those above:
\begin{enumerate}
\item[\stalkF] The \emph{stalks} of the $\O$-perverse sheaves $\IC_w(\O)$ are torsion-free. Equivalently, the stalks of the $\F$-perverse sheaves $\F \lotimes \IC_w(\O)$ enjoy a parity-vanishing property.
\item[\kosF] The category $\Perv_{(B)}(\cB,\F)$ admits a \emph{standard $Q$-Koszul} grading.
\end{enumerate}
The definition of a standard $Q$-Koszul category---a generalization of the ordinary Koszul property, due to Parshall--Scott~\cite{ps}---will be recalled in~\S\ref{ss:qkoszul}.
The status of these conditions in various examples will be discussed at the end of~\S\ref{ss:pervmix}.

One of the main results of
this paper is that statements~\stalkF{} and \kosF{} are nearly equivalent to each other.  Statement~\stalkF{} may be compared to (and was inspired by) the Mirkovi\'c--Vilonen conjecture~\cite{mv} (now a theorem in most cases~\cite{arider}), which asserts that spherical $\IC$-sheaves on the affine Grassmannian have torsion-free stalks.  Statement~\kosF{} is closely related to certain conjectures of Cline, Parshall, and Scott~\cite{cps,ps} on representations of algebraic groups.

\subsection{Mixed modular perverse sheaves}
\label{ss:pervmix}

In the characteristic zero case, statements~\stalkC{} and~\kosC{} are best understood in the framework of mixed $\Qlb$-sheaves. In~\cite{modrap2} we defined and studied a replacement for these objects in the modular context (when $\ell$ is good for $G$). More precisely, for $\E=\K$, $\O$, or $\F$ we defined a triangulated category $\Dmix_{(B)}(\cB,\E)$, endowed with a ``Tate twist'' $\la 1 \ra$ and a ``perverse t-structure'' whose heart we denote by $\Perv^\mix_{(B)}(\cB,\E)$. This category is also endowed with a t-exact ``forgetful'' functor $\Dmix_{(B)}(\cB,\E) \to \Db_{(B)}(\cB,\E)$, where the usual Bruhat-constructible derived category $\Db_{(B)}(\cB,\E)$ is endowed with the usual perverse t-structure. The main tool in this construction is the category $\Parity_{(B)}(\cB,\E)$ of parity complexes on $\cB$ in the sense of Juteau--Mautner--Williamson~\cite{jmw}. The indecomposable objects in the latter category are naturally parametrized by $W \times \Z$; we denote as usual by $\cE_w$ the object associated with $(w,0)$.

The category $\Perv^\mix_{(B)}(\cB,\F)$ is a graded quasihereditary category, and can be considered a ``graded version'' of the category $\Perv_{(B)}(\cB,\F)$. The analogue of this category when $\F$ is replaced by $\K$ can be identified with the category studied in~\cite[\S 4.4]{bgs}, and is known to be Koszul (and even \emph{standard Koszul}). One might wonder if the category $\Perv_{(B)}(\cB,\F)$ enjoys a similar property, or some weaker analogues. The main theme of this paper is to relate these properties to properties of the usual perverse sheaves on $\cB$ or the flag variety $\cBv$ of the Langlands dual reductive group. More precisely, we consider the following four properties:
\begin{enumerate}
\item
The category $\Perv^\mix_{(B)}(\cB,\F)$ is positively graded.\label{it:pmix-pos}
\item
The category $\Perv^\mix_{(B)}(\cB,\F)$ is standard $Q$-Koszul.\label{it:pmix-qkos}
\item
The category $\Perv^\mix_{(B)}(\cB,\F)$ is metameric.\label{it:pmix-metameric}
\item
The category $\Perv^\mix_{(B)}(\cB,\F)$ is standard Koszul.\label{it:pmix-kos}
\end{enumerate}
Here, condition~\eqref{it:pmix-pos} is a natural condition defined and studied in~\S\ref{ss:nonneg}. As explained above, condition~\eqref{it:pmix-qkos}---which is stronger than~\eqref{it:pmix-pos}---was introduced by Parshall--Scott~\cite{ps}; see~\S\ref{ss:qkoszul}. Condition~\eqref{it:pmix-metameric}---which is also stronger than~\eqref{it:pmix-pos} but unrelated to~\eqref{it:pmix-qkos} a priori---is a technical condition defined and studied in~\S\ref{ss:metameric}. Condition~\eqref{it:pmix-kos} is the standard condition studied e.g.~in~\cite{adl, mazorchuk}; see also~\cite{bgs}. This condition is stronger than~\eqref{it:pmix-metameric} and ~\eqref{it:pmix-qkos}.

Our main result can be stated as follows. (Here, $\cEv_w$, resp.~$\ICv_w$, is the parity sheaf, resp.~$\IC$-sheaf on $\cBv$ naturally associated with $w$. This statement combines parts of Theorems~\ref{thm:main1}, \ref{thm:main2}, and~\ref{thm:koszulity}.)

\begin{thm*}
\label{thm:main-intro}
Assume that $\ell$ is good for $G$.

\begin{enumerate}
\item 
\label{it:thm-pos}
The following conditions are equivalent:
\begin{enumerate}
\item
The category $\Perv^\mix_{(B)}(\cB,\F)$ is positively graded.
\item
For all $w \in W$, the parity sheaf $\cEv_w(\F)$ on $\cBv$ is perverse.
\end{enumerate}
\item 
\label{it:thm-qkos}
The following conditions are equivalent:
\begin{enumerate}
\item
The category $\Perv^\mix_{(B)}(\cB,\F)$ is metameric.\label{it:pwb-metameric}
\item
The category $\Perv^\mix_{(\Bv)}(\cBv,\F)$ is standard $Q$-Koszul.
\item
For all $w \in W$, the parity sheaf $\cE_w(\F)$ is perverse, and the $\O$-perverse sheaf $\ICv_w(\O)$ on $\cBv$ has torsion-free stalks.\label{it:ew-tor}
\end{enumerate}
\item 
\label{it:thm-kos}
The following conditions are equivalent:
\begin{enumerate}
\item
The category $\Perv^\mix_{(B)}(\cB,\F)$ is standard Koszul.
\item
For all $w \in W$ we have $\cE_w(\F) \cong \IC_w(\F)$.
\item For all $w \in W$, the $\O$-perverse sheaf $\IC_w(\O)$ on $\cB$ has torsion-free stalks and costalks.\label{it:ew-ic}
\end{enumerate}
Moreover, these statements hold if and only if the analogous statements for the Langlands dual group hold.
\end{enumerate}
\end{thm*}

In this theorem, part~\eqref{it:thm-pos} is an immediate consequence of the results of~\cite{modrap2}. Part~\eqref{it:thm-kos} is also not difficult to prove.  However, as noted earlier, Williamson~\cite{williamson} has exhibited counterexamples to condition~\eqref{it:ew-ic} in groups of arbitrarily large rank.

Part~\eqref{it:thm-qkos} of the theorem is the most interesting and delicate case, and its proof requires the introduction of new tools.  Williamson has informed us that condition~\eqref{it:ew-tor} holds for $G = \mathrm{GL}(n)$ with $n \le 9$ in all characteristics.  His counterexamples to~\eqref{it:ew-ic} all involve torsion only in the costalks of the $\IC_w(\O)$, not in their stalks.  Thus, as of this writing, there are no known counterexamples to the conditions in part~\eqref{it:thm-qkos}.

\subsection{Weights}
\label{ss:intro-weights}

To prove part~\eqref{it:thm-qkos} of the theorem above, we introduce a formalism which plays a role similar to Deligne's theory of weights for mixed $\Qlb$-perverse sheaves. (However, it is much less powerful than Deligne's theory: in particular, the existence of a ``weight filtration'' on mixed modular perverse sheaves is not automatic.) More precisely, in~\S\ref{ss:weights} we define what it means for an object of $\Dmix_{(B)}(\cB,\E)$ to have weights${}\leq n$ or${}\geq n$, and we prove that the $!$- and $*$-pullback and pushforward functors associated with locally closed inclusions of unions of Bruhat cells enjoy the same stability properties for this formalism as in the case of mixed $\Qlb$-sheaves (cf.~\cite[Stabilit\'es~5.1.14]{bbd}).

Next, in~\S\ref{ss:perv-circ}, we use the theory of weights to define a new, smaller abelian category $\Perv^\circ_{(B)}(\cB,\E) \subset \Dmix_{(B)}(\cB,\E)$.  This is \emph{not} the heart of a t-structure on $\Dmix_{(B)}(\cB,\E)$; for instance, when $\E = \Qlb$, it is the category consisting of semisimple pure perverse sheaves of weight $0$. The category $\Perv^\circ_{(B)}(\cB,\F)$ need not be semisimple, but it is always quasihereditary, so one may speak of standard and costandard objects in $\Perv^\circ_{(B)}(\cB,\F)$. These objects are parametrized by $W$, and the standard, resp.~costandard, object associated with $w$ is denoted $\Delta^\circ_w(\F)$, resp.~$\nabla^\circ_w(\F)$.  A careful study of the structure of the $\Delta^\circ_w(\F)$, carried out in~\S\ref{ss:stalks-deltao}, is the glue linking the various assertions in part~\eqref{it:thm-qkos} of the theorem.  


\subsection{Interpreting the $\Delta^\circ_w(\F)$}
\label{ss:interpret}

In the course of the proof, we will see that if $\Perv_{(B)}(\cB,\F)$ is positively graded, then $\Delta^\circ_w(\F) \cong \F \lotimes \IC^\mix_w(\O)$. This property is analogous to the fact~\cite[\S8]{mv} that in the category of spherical perverse sheaves on the affine Grassmannian,  standard objects are of the form $\F \lotimes \IC_\lambda(\O)$.  Of course, in the setting of~\cite{mv}, there is a representation-theoretic interpretation for these objects as well: they correspond to Weyl modules under the geometric Satake equivalence.

If one hopes to prove that the conditions in part~\eqref{it:thm-qkos} of the theorem are actually true, it will likely be useful to find a representation-theoretic interpretation of the $\Delta^\circ_w(\F)$.  One candidate is the class of \emph{reduced standard modules} introduced by Cline--Parshall--Scott~\cite{cps}. These are certain representations of an algebraic group, obtained by modular reduction of irreducible quantum group representations.  It is likely that under the equivalence of~\cite[Theorem~2.4]{modrap1}, reduced standard modules correspond to objects of the form $\F \lotimes \IC_w(\O)$. With this in mind, condition~\eqref{it:pwb-metameric} should be compared to~\cite[Conjecture~6.5]{cps}, and condition~\eqref{it:ew-tor} to~\cite[Conjecture~6.2]{cps}.  (See~\cite{ps,ps:qka} for other results about standard $Q$-Koszulity in the context of representations of algebraic groups.)

There are further parallels between $\Perv^\circ_{(B)}(\cB,\F)$ and the affine Grassmannian that may lead to future insights.  We have already noted that condition~\eqref{it:ew-tor} resembles the Mirkovi\'c--Vilonen conjecture.  In fact, a version of the meta\-meric property (see~\cite[Corollary~5.1.13]{bk}) plays a role in the proof of that conjecture. Separately, the conditions in part~\eqref{it:thm-qkos} imply that the $\cEv_w(\F)$ are precisely the tilting objects in $\Perv^\circ_{(\Bv)}(\cBv,\F)$.  This is similar to the main result of~\cite{jmw2}, which relates spherical parity sheaves to tilting modules via the geometric Satake equivalence.


\subsection{Acknowledgements}
\label{ss:acknowledgements}

We thank Geordie Williamson for stimulating discussions.

\subsection{Contents}
\label{ss:contents}

Section~\ref{sec:homological} contains general results on positively graded quasihereditary categories, including metameric and standard $Q$-Koszul categories.  In Sections~\ref{sec:weights} and~\ref{sec:ico}, we work in the general setting of a stratified variety satisfying the assumptions of~\cite[\S\S2--3]{modrap2}.  These sections develop the theory of weights for $\Dmix_\scS(X,\F)$, and contain the definition of $\Perv^\circ_\scS(X,\F)$.  Finally, in Section~\ref{sec:koszulity} we concentrate on the case of flag varieties, and prove our main theorems.

\section{Positivity conditions for graded quasihereditary categories}
\label{sec:homological}

Throughout this section, $\bk$ will be a field, and $\cA$ will be a finite-length $\bk$-linear abelian category.

\subsection{Graded quasihereditary categories}
\label{ss:qshcategories}

We begin by recalling the definition of graded quasihereditary categories. We refer to~\cite[Appendix~A]{modrap2} for reminders on the main properties of these categories.

Assume $\cA$ is equipped with an automorphism $\la 1\ra: \cA \to \cA$.  Let $\Irr(\cA)$ be the set of isomorphism classes of irreducible objects of $\cA$, and let $\scS = \Irr(\cA)/\Z$, where $n \in \Z$ acts on $\Irr(\cA)$ by $\la n\ra$.  Assume that $\scS$ is equipped with a partial order $\le$, and that for each $s \in \scS$, we have a fixed representative simple object $\Lgr_s$. Assume also we are given, for any $s \in \scS$, objects $\dgr_s$ and $\ngr_s$, and morphisms $\dgr_s \to \Lgr_s$ and $\Lgr_s \to \ngr_s$. For $\scT \subset \scS$, we denote by $\cA_{\scT}$ the Serre subcategory of $\cA$ generated by the objects $\Lgr_t \la n \ra$ for $t \in \scT$ and $n \in \Z$. We write $\cA_{\leq s}$ for $\cA_{\{t \in \scS \mid t \leq s\}}$, and similarly for $\cA_{<s}$.

\begin{defn}\label{defn:qhered}
The category $\cA$ (with the data above) is said to be \emph{graded quasihereditary} if the following conditions hold:
\begin{enumerate}
\item The set $\scS$ is finite.\label{it:qh-def-fin}
\item For each $s \in \scS$, we have 
\[
\Hom(\Lgr_s,\Lgr_s \la n \ra) = \begin{cases}
\bk & \text{if $n=0$;} \\
0 & \text{otherwise.}
\end{cases}
\]
 \label{it:qh-def-split}
\item The kernel of $\dgr_s \to \Lgr_s$ and the cokernel of $\Lgr_s \to \ngr_s$ belong to $\cA_{<s}$.\label{it:qh-def-ker}
\item For any closed subset $\scT \subset \scS$ (in the order topology), if $s \in \scT$ is maximal, then $\dgr_s \to \Lgr_s$ is a projective cover in $\cA_{\scT}$, and $\Lgr_s \to \ngr_s$ is an injective envelope in $\cA_{\scT}$.\label{it:qh-def-cover}
\item We have $\Ext^2(\dgr_s, \ngr_t\la n\ra) = 0$ for all $s, t \in \scS$ and $n \in \Z$.\label{it:qh-def-ext2}
\end{enumerate}
\end{defn}

Recall (see~\cite[Theorem~A.3]{modrap2}) that if $\cA$ is graded quasihereditary then it has enough projective objects, and that each projective object admits a standard filtration, i.e.~a filtration with subquotients of the form $\dgr_t \la n \ra$ ($t \in \scS$, $n \in \Z$). Moreover, if we denote by $P_s^\gr$ the projective cover of $\Lgr_s$, then a graded form of the reciprocity formula holds: 
\begin{equation}
\label{eqn:reciprocity}
(P^\gr_s : \dgr_t \la n \ra) = [\ngr_t \la n \ra : \Lgr_s],
\end{equation}
where the left-hand side denotes the multiplicity of $\dgr_t \la n \ra$ in any standard filtration of $P^\gr_s$, and the right-hand side denotes the usual multiplicity as a composition factor. Similar claims hold for injective objects.

Below we will also consider some (ungraded) quasihereditary categories: these are categories satisfying obvious analogues of the conditions in Definition~\ref{defn:qhered}.

Later we will need the following properties.

\begin{lem}
\label{lem:qher-closed-subset}
Let $\scT \subset \scS$ be a closed subset. 
\begin{enumerate}
\item
The subcategory $\cA_{\scT} \subset \cA$ is a graded quasihereditary category, with standard (resp.~costandard) objects $\dgr_t$ (resp.~$\ngr_t$) for $t \in \scT$. Moreover, the functor $\iota_{\scT}: \Db \cA_{\scT} \to \Db \cA$ induced by the inclusion $\cA_{\scT} \subset \cA$ is fully faithful.
\item
The Serre quotient $\cA / \cA_{\scT}$ is a graded quasihereditary category for the order on $\scS \smallsetminus \scT$ obtained by restriction from the order on $\scS$. The standard (resp.~costandard) objects are the images in the quotient of the objects $\dgr_s$ (resp.~$\ngr_s$) for $s \in \scS \smallsetminus \scT$.
\item
The natural functor $\Db(\cA) / \Db(\cA_{\scT}) \to \Db(\cA/\cA_{\scT})$ (where the left-hand side is the Verdier quotient) is an equivalence. Moreover, the functors $\Pi_{\scT} : \Db(\cA) \to \Db(\cA/\cA_{\scT})$ and $\iota_{\scT}$ admit left and right adjoints, denoted $\Pi_{\scT}^R$, $\Pi_{\scT}^L$, $\iota_{\scT}^R$, $\iota_{\scT}^L$, which satisfy
\begin{equation}
\label{eqn:adjoints-pi-d-n}
\Pi_\scT^L \circ \Pi_{\scT}(\dgr_s) \cong \dgr_s, \qquad \Pi_\scT^R \circ \Pi_{\scT}(\ngr_s) \cong \ngr_s
\end{equation}
and such that, for any $M$ in $\Db(\cA)$, the adjunction morphisms induce functorial triangles
\[
\iota_{\scT} \iota_{\scT}^R M \to M \to \Pi_\scT^R \Pi_\scT M \xrightarrow{[1]}, \qquad \Pi_\scT^L \Pi_\scT M \to M \to \iota_{\scT} \iota_{\scT}^L M \xrightarrow{[1]}.
\]
\end{enumerate}
\end{lem}

\begin{proof}
(1) It is clear that $\cA_{\scT}$ satisfies the first four conditions in Definition~\ref{defn:qhered}. To check that it satisfies the fifth condition, one simply observes that the natural morphism $\Ext^2_{\cA_{\scT}}(\dgr_s, \ngr_t\la n\ra) \to \Ext^2_{\cA}(\dgr_s, \ngr_t\la n\ra)$ is injective for $s,t \in \scT$, $n \in \Z$, see e.g.~\cite[Lemma 3.2.3]{bgs}. Since the second space is trivial by assumption, the first one is trivial also.

Now it follows from the definitions that the category $\Db \cA_{\scT}$ is generated (as a triangulated category) by the objects $\dgr_t \la n \ra$ for $t \in \scT$ and $n \in \Z$, as well as by the objects $\ngr_t \la n \ra$ for $t \in \scT$ and $n \in \Z$. Hence, by a standard argument, to prove that $\iota_{\scT}$ is fully faithful, it is enough to prove that for $s,t \in \scT$ and $k,n \in \Z$ the natural morphism
\[
\Ext^k_{\cA_{\scT}}(\dgr_s, \ngr_t \la n \ra) \to \Ext^k_{\cA}(\dgr_s, \ngr_t \la n \ra)
\]
is an isomorphism. However in both categories $\cA$ and $\cA_{\scT}$ we have
\[
\Ext^k(\dgr_s, \ngr_t \la n \ra) = \begin{cases}
\bk & \text{if $s=t$, $k=n=0$;} \\
0 & \text{otherwise,}
\end{cases}
\]
see e.g.~\cite[Equation~(A.1)]{modrap2}. Hence this claim is clear.

(2) It is clear that the quotient $\cA/\cA_{\scT}$ satisfies conditions~\eqref{it:qh-def-fin}, \eqref{it:qh-def-split}, and~\eqref{it:qh-def-ker} of Definition~\ref{defn:qhered}. To check that it satisfies condition~\eqref{it:qh-def-cover}, we denote by $\pi_{\scT} : \cA \to \cA_{\scT}$ the quotient morphism. Then one can easily check that if $s \in \scS \smallsetminus \scT$, for any $M$ in $\cA$ the morphisms
\begin{align*}
\Hom_{\cA}(\dgr_s, M) & \to \Hom_{\cA/\cA_{\scT}}(\pi_{\scT}(\dgr_s), \pi_{\scT}(M)), \\ 
\Hom_{\cA}(M, \ngr_s) & \to \Hom_{\cA/\cA_{\scT}}(\pi_{\scT}(M), \pi_{\scT}(\ngr_s))
\end{align*}
induced by $\pi_{\scT}$ are isomorphisms. Using~\cite[Corollaire~3 on p.~369]{gabriel}, one easily deduces that condition~\eqref{it:qh-def-cover} holds.

To prove condition~\eqref{it:qh-def-ext2}, we observe that, by~\cite[Corollaire~1 on p.~375]{gabriel}, the subcategory $\cA_{\scT}$ is localizing; by~\cite[Corollaire~2 on p.~375]{gabriel} we deduce that $\cA/\cA_{\scT}$ has enough injectives, and that every injective object is of the form $\pi_{\scT}(I)$ for some $I$ injective in $\cA$. In particular, since $\pi_{\scT}(\ngr_s)$ is either $0$ or a costandard object of $\cA/\cA_{\scT}$, we deduce that injective objects in $\cA/\cA_{\scT}$ admit costandard filtrations. By a standard argument (see e.g.~\cite[Corollary~3]{ringel}), this implies condition~\eqref{it:qh-def-ext2}.

(3) Observe that the objects $\{\dgr_s, s \in \scS\}$ form a graded exceptional set in $\Db(\cA)$ in the sense of~\cite[\S 2.1.5]{bezru2}. Hence, applying the general theory of these sequences developed in \cite{bezru, bezru2} we find that $\iota_\scT$ and the quotient functor $' \Pi_\scT : \Db(\cA) \to \Db(\cA) / \Db(\cA_{\scT})$ admit left and right adjoints, which induce functorial triangles as in the lemma. If we denote by $' \Pi_\scT^L$ (resp.~$' \Pi_\scT^R$) the left (resp.~right) adjoint to $' \Pi_\scT$, it is easily checked that we have 
\[
(' \Pi_\scT^L) \circ (\Pi_{\scT})(\dgr_s) \cong \dgr_s \quad \text{and} \quad (' \Pi_\scT^R) \circ (\Pi_{\scT})(\ngr_s) \cong \ngr_s
\]
for any $s \in \scS \smallsetminus \scT$ (see e.g.~\cite[Lemma~4(d)]{bezru} for a similar claim). Using this property and an argument similar to the one used to prove that $\iota_\scT$ is fully faithful, one can deduce that the natural functor $\Db(\cA) / \Db(\cA_{\scT}) \to \Db(\cA/\cA_{\scT})$ is an equivalence, which finishes the proof.
\end{proof}

\subsection{Positively graded quasihereditary categories}
\label{ss:nonneg}

In this section we will mainly consider graded quasihereditary categories which exhibit some positivity properties. The precise definition is as follows.

\begin{defn}
\label{defn:qh-pos}
Let $\cA$ be a graded quasihereditary category.  We say that $\cA$ is \emph{positively graded} if for all $s, t \in \scS$, we have $[ P^\gr_s : \Lgr_t\la n\ra ] = 0$ whenever $n > 0$.
\end{defn}

\begin{rmk}
\label{rmk:qh-pos-ring}
The condition in Definition~\ref{defn:qh-pos} is equivalent to requiring that we have $\Hom(P^\gr_t, P^\gr_s\la n\ra) = 0$ whenever $n < 0$.  In other words, if we let $P^\gr = \bigoplus_{s\in \scS} P^\gr_s$, then $\cA$ is positively graded if and only if the graded ring
\[
R:= \bigoplus_{n \in \Z} \Hom(P^\gr, P^\gr\la n\ra)
\]
is concentrated in nonnegative degrees. Note that $R$ is a finite dimensional $\bk$-algebra, and that the functor $M \mapsto \bigoplus_n \Hom_{\cA}(P^\gr, M \la n \ra)$ induces an equivalence of categories between $\cA$ and the category of finite dimensional graded right $R$-modules.
\end{rmk}

\begin{prop}\label{prop:pos-qher}
Let $\cA$ be a graded quasihereditary category.  The following conditions are equivalent:
\begin{enumerate}
\item $\cA$ is positively graded.\label{it:pos-qher}
\item We have $[\dgr_s : \Lgr_t\la n\ra] = (P^\gr_s : \dgr_t\la n\ra) = 0$ whenever $n > 0$.\label{it:pos-delta-proj}
\item We have $[\dgr_s : \Lgr_t\la n\ra] = [\ngr_s \la n \ra : \Lgr_t] = 0$ whenever $n > 0$.\label{it:pos-delta-nabla}
\item We have $\Ext^1(\Lgr_s, \Lgr_t\la n\ra) = 0$ for $n > 0$.\label{it:pos-ext}
\item Every object $M \in \cA$ admits a canonical filtration $W_\bullet M$ with the property that every composition factor of $\Gr^W_i M$ is of the form $\Lgr_s\la i\ra$, and every morphism in $\cA$ is strictly compatible with this filtration.\label{it:pos-wt-filt}
\end{enumerate}
\end{prop}
\begin{proof}
\eqref{it:pos-qher}${}\Longrightarrow{}$\eqref{it:pos-delta-proj}.  Since $\dgr_s$ is a quotient of $P^\gr_s$, we clearly have $[\dgr_s : \Lgr_t\la n\ra ] = 0$ for $n > 0$.  If we had $(P^\gr_s : \dgr_t\la n\ra) \ne 0$ for some $s, t$ and some $n > 0$, then we would also have $(P^\gr_s : \Lgr_t\la n\ra ) \ne 0$, contradicting the assumption.

\eqref{it:pos-delta-proj}${}\Longrightarrow{}$\eqref{it:pos-qher}. This is obvious.

The equivalence \eqref{it:pos-delta-proj}${}\Longleftrightarrow{}$\eqref{it:pos-delta-nabla} follows from the reciprocity formula~\eqref{eqn:reciprocity}.

\eqref{it:pos-qher}${}\Longrightarrow{}$\eqref{it:pos-ext}.  Let $K$ be the kernel of $P^\gr_s \to \Lgr_s$.  Note that if $n > 0$, then $[K  : \Lgr_t\la n\ra] = 0$, and hence $\Hom(K, \Lgr_t\la n\ra) = 0$.  We deduce the desired result from the exact sequence
\[
\cdots \to \Hom(K,\Lgr_t\la n\ra) \to \Ext^1(\Lgr_s,\Lgr_t\la n\ra) \to \Ext^1(P^\gr_s, \Lgr_t\la n\ra) \to \cdots.
\]

\eqref{it:pos-ext}${}\Longrightarrow{}$\eqref{it:pos-wt-filt}.  This follows from~\cite[Lemme~5.3.6]{bbd} (see also the proof of~\cite[Th\'eor\`eme~5.3.5]{bbd}).

\eqref{it:pos-wt-filt}${}\Longrightarrow{}$\eqref{it:pos-qher}.  Consider the weight filtration $W_\bullet P^\gr_s$ of $P^\gr_s$.  Let $n$ be the largest integer such that $\Gr^W_n P^\gr_s \ne 0$.  Then $\Gr^W_n P^\gr_s$ is a quotient of $P^\gr_s$, and in particular, $P^\gr_s$ has a quotient of the form $\Lgr_t\la n\ra$.  But $\Lgr_s$ is the unique simple quotient of $P^\gr_s$, so we must have $n = 0$, and the result follows.
\end{proof}

Let us note the following consequence of Proposition~\ref{prop:pos-qher}, which is immediate from condition~\eqref{it:pos-delta-nabla} of the proposition.

\begin{cor}
\label{cor:quotient-pos}
If $\cA$ is a positively graded quasihereditary category and if $\scT \subset \scS$ is closed, then the graded quasihereditary category $\cA/\cA_{\scT}$ is positively graded.
\end{cor}

It is easy to see 
that in a positively graded quasihereditary category, any $\Lgr_s$ admits a projective resolution whose terms are direct sums of various $P^\gr_t\la n\ra$ with $n \le 0$.  As a consequence, for all $k \ge 0$ we have
\begin{equation}\label{eqn:pos-higher-ext}
\Ext^k(\Lgr_s, \Lgr_t\la n\ra) = 0 \qquad\text{for $n > 0$.}
\end{equation}

\begin{prop}[cf.~{\cite[Proposition~3.1(a)]{ps}}]\label{prop:parshall-scott}
Let $\cA$ be a positively graded quasihereditary category, and let $\cA^\circ$ be the Serre subcategory generated by the simple objects $\{ \Lgr_s \mid s \in \scS \}$ (i.e., without Tate twists).  Then $\cA^\circ$ is a quasihereditary category (with weight poset $\scS$), with standard and costandard objects given respectively by
\[
\Delta^\circ_s := \Gr^W_0\dgr_s
\qquad\text{and}\qquad
\nabla^\circ_s := \Gr^W_0\ngr_s.
\]
\end{prop}

\begin{proof}
It is clear that $\cA^0$ is a finite length category, and that its simple objects are parametrized by $\scS$. It is also clear from the definitions that $\Delta^\circ_s$ is a quotient of $\dgr_s$, and that the surjection $\dgr_s \to \Lgr_s$ factors through a surjection $\Delta^\circ_s \to \Lgr_s$. Similarly, $\nabla^\circ_s$ is a subobject of $\ngr_s$, and the injection $\Lgr_s \to \ngr_s$ factors through an injection $\Lgr_s \to \nabla^\circ_s$.
The ungraded analogues of axioms~\eqref{it:qh-def-fin}, \eqref{it:qh-def-split} and~\eqref{it:qh-def-ker} of Definition~\ref{defn:qhered} are clear.

We now turn to axiom~\eqref{it:qh-def-cover}.
Since $\Delta^\circ_s$ is a quotient of $\dgr_s$, it has a unique simple quotient, isomorphic to $\Lgr_s$.  Next, let $\scT \subset \scS$ be closed, with $s$ maximal in $\scT$. For $t \in \scT$, consider the exact sequence
\[
\cdots \to \Hom_{\cA_{\scT}}(W_{-1}\dgr_s,\Lgr_t) \to \Ext^1_{\cA_{\scT}}(\Delta^\circ_s, \Lgr_t) \to \Ext^1_{\cA_{\scT}}(\dgr_s, \Lgr_t) \to \cdots.
\]
The first term vanishes because $W_{-1}\dgr_s$ has only composition factors of the form $\Lgr_t \la n \ra$ with $n<0$, and the last term vanishes by axiom~\eqref{it:qh-def-cover} for $\cA$. So the middle term does as well. It clear that $\Ext^1_{\cA^\circ_{\scT}}(\Delta^\circ_s, \Lgr_t)=\Ext^1_{\cA_{\scT}}(\Delta^\circ_s, \Lgr_t)$, so we have shown that $\Delta^\circ_s$ is a projective cover of $\Lgr_s$ in $\cA^\circ_{\scT}$.

A similar argument shows that $\nabla^\circ_s$ is an injective envelope of $\Lgr_s$ in $\cA^\circ_{\scT}$; we omit further details.

Finally, we consider the analogue of axiom~\eqref{it:qh-def-ext2}. Consider the exact sequence
\[
\cdots \to \Ext^1_\cA(W_{-1}\dgr_s, \ngr_t) \to \Ext^2_\cA(\Delta^\circ_s, \ngr_t) \to \Ext^2_\cA(\dgr_s, \ngr_t) \to \cdots.
\]
The first term vanishes by Proposition~\ref{prop:pos-qher}\eqref{it:pos-ext}, and the last by axiom~\eqref{it:qh-def-ext2} for $\cA$, so the middle term does as well.  That term is also the last term in the exact sequence
\[
\cdots \to \Ext^1_\cA(\Delta^\circ_s, \ngr_t / W_0 \ngr_t) \to \Ext^2_\cA(\Delta^\circ_s, \nabla^\circ_t) \to \Ext^2_\cA(\Delta^\circ_s, \ngr_t) \to \cdots,
\]
whose first term again vanishes by Proposition~\ref{prop:pos-qher}\eqref{it:pos-ext}.  We have now shown that $\Ext^2_\cA(\Delta^\circ_s, \nabla^\circ_t) = 0$.  By a standard argument (see e.g.~\cite[Lemma~3.2.3]{bgs}), the natural map $\Ext^2_{\cA^\circ}(\Delta^\circ_s, \nabla^\circ_t) \to \Ext^2_\cA(\Delta^\circ_s, \nabla^\circ_t)$ is injective, so the former vanishes as well, as desired.
\end{proof}

\begin{rmk}
With the notation of Remark~\ref{rmk:qh-pos-ring}, if $\cA$ is a positively graded quasihereditary category, then the category $\cA^\circ$ identifies with the subcategory of the category of finite-dimensional graded right $R$-modules consisting of modules concentrated in degree $0$; in other words, with the category of finite-dimensional right modules over the $0$-th part $R^0$ of $R$.
\end{rmk}

The determination of $\Ext^2_\cA(\Delta^\circ_s,\nabla^\circ_t)$ at the end of the preceding proof can easily be adapted to higher $\Ext$-groups: by using~\eqref{eqn:pos-higher-ext} in place of Proposition~\ref{prop:pos-qher}\eqref{it:pos-ext}, and~\cite[Eq.~(A.1)]{modrap2} in place of axiom~\eqref{it:qh-def-ext2} for $\cA$, we find that
\begin{equation}\label{eqn:qher-pos-ff}
\Ext^k_\cA(\Delta^\circ_s, \nabla^\circ_t) = 0 \qquad\text{for all $k \ge 1$.}
\end{equation}
As in Lemma~\ref{lem:qher-closed-subset}, this implies the following fact.

\begin{lem}\label{lem:qher-pos-ff}
Let $\cA$ be a positively graded quasihereditary category.  The natural functor $\Db\cA^\circ \to \Db\cA$ is fully faithful.
\end{lem}

\subsection{Metameric categories}
\label{ss:metameric}

We have seen above that any positively graded quasihereditary category contains two classes of objects worthy of being called ``standard'': the usual $\dgr_s$, and the new $\Delta^\circ_s$.  In this subsection, we study categories in which these two classes are closely related.

\begin{defn}\label{defn:metameric}
Let $\cA$ be a positively graded quasihereditary category. We say that $\cA$ is a \emph{metameric category} if for all $s \in \scS$ and all $i \in \Z$, the object $\bigl( \Gr^W_i \dgr_s \bigr) \la -i\ra \in \cA^\circ$ admits a standard filtration, and $\bigl( \Gr^W_i \ngr_s \bigr) \la -i \ra \in \cA^\circ$ admits a costandard filtration.
\end{defn}

This term is borrowed from biology, where \emph{metamerism} refers to a body plan containing repeated copies of some smaller structure.  The analogy is that in our setting, each $\Delta_s$ is made up of copies of the smaller objects $\Delta^\circ_u$.

\begin{thm}\label{thm:qher-wdelta}
Let $\cA$ be a metameric category.  
For any $s \in \scS$, there exists a unique object $\wDelta^\gr_s \in \cA$ which satisfies the following properties.
\begin{enumerate}
\item $\wDelta^\gr_s$ has a unique simple quotient, isomorphic to $\Lgr_s$.\label{it:qher-head}
\item For all $r \in \scS$ and $k \in \Z_{>0}$, we have $\Ext^k(\wDelta^\gr_s, \Lgr_r) = 0$ if $r \le s$.\label{it:qher-ext0}
\item For all $r \in \scS$ and $k \in \Z_{>0}$, we have $\Ext^k(\wDelta^\gr_s, \Lgr_r\la n\ra) = 0$ if $n \ne 0$.\label{it:qher-ext-neg}
\item There is a surjective map $\wDelta^\gr_s \to \dgr_s$ whose kernel admits a filtration whose subquotients are various $\dgr_u\la n\ra$ with $u > s$ and $n < 0$.\label{it:qher-filt}
\end{enumerate}
Dually, for any $s \in \scS$, there exists a unique object $\wnabla^\gr_s \in \cA$ which satisfies the following properties.
\begin{enumerate}
\item[(1$'$)] $\wnabla^\gr_s$ has a unique simple subobject, isomorphic to $\Lgr_s$.
\item[(2$'$)] For all $r \in \scS$ and $k \in \Z_{>0}$, we have $\Ext^k(\Lgr_r, \wnabla^\gr_s) = 0$ if $r \le s$.
\item[(3$'$)] For all $r \in \scS$ and $k \in \Z_{>0}$, we have $\Ext^k(\Lgr_r\la n\ra,\wnabla^\gr_s) = 0$ if $n \ne 0$.
\item[(4$'$)] There is an injective map $\ngr_s \to \wnabla^\gr_s$ whose cokernel admits a filtration whose subquotients are various $\ngr_u\la n\ra$ with $u > s$ and $n > 0$.
\end{enumerate}
Conversely, if $\cA$ is a positively graded quasihereditary category which contains objects $\wDelta^\gr_s$ and $\wnabla^\gr_s$ satisfying the above properties for all $s \in \scS$, then $\cA$ is metameric.
\end{thm}

\begin{proof}
Suppose that $\cA$ is metameric.  We first remark that, if $\wDelta^\gr_s$ exists, then it is the projective cover of $\Lgr_s$ in the Serre subcategory of $\cA$ generated by the objects $\Lgr_r$ with $r \leq s$ and the objects $\Lgr_t \la n \ra$ for all $t \in \scS$ and $n \neq 0$. Hence uniqueness is clear. It also follows from this remark that the map in~\eqref{it:qher-filt} is the unique (up to scalar) nonzero morphism $\wDelta^\gr_s \to \dgr_s$. 

To prove existence,
we can assume without loss of generality that $\scS$ is the set $\{1,2, \ldots, N\}$ (with its natural order).  We proceed by induction on $N$.  Let $\cA':=\cA_{\leq N-1}$, and 
assume the theorem is known to hold for $\cA'$.  For each $i \le N-1$, let $\wDelta^{\gr\prime}_i$ be the object in $\cA'$ satisfying the properties of Theorem~\ref{thm:qher-wdelta} for $\cA'$.  

We begin by constructing the objects $\wDelta^\gr_i$.  For $i = N$, we simply set
\[
\wDelta^\gr_N := \dgr_N.
\]
This object clearly has properties~\eqref{it:qher-head}--\eqref{it:qher-filt}. Now suppose $i < N$.  For $n < 0$, let $E_n := \Ext^1(\wDelta^{\gr\prime}_i, \dgr_N\la n\ra)$.  Let $\epsilon_n$ be the canonical element of $E_n^* \otimes E_n \cong \Ext^1(\wDelta^{\gr\prime}_i, E_n^* \otimes \dgr_N\la n\ra)$, and let
\[
\epsilon := \bigoplus_{n < 0} \epsilon_n \in \Ext^1\left( \wDelta^{\gr\prime}_i, \bigoplus_{n<0} E_n^* \otimes \dgr_N\la n\ra\right).
\]
(Note that only finitely many of the spaces $E_n$ are nonzero, so these direct sums are finite.)  Define $\wDelta^\gr_i$ to be the middle term of the corresponding short exact sequence:
\begin{equation}\label{eqn:wdelta-defn}
0 \to \bigoplus_{n < 0} E_n^* \otimes \dgr_N\la n\ra \to \wDelta^\gr_i \to \wDelta^{\gr\prime}_i \to 0.
\end{equation}
Then, for any $m < 0$, the natural map
\begin{equation}\label{eqn:qher-can}
\Hom \bigl( \bigoplus_{n < 0} E_n^* \otimes \dgr_N\la n\ra, \dgr_N\la m\ra \bigr) \to \Ext^1(\wDelta^{\gr\prime}_i, \dgr_N\la m\ra)
\end{equation}
is an isomorphism.  For brevity, we henceforth write $C := \bigoplus_{n < 0} E_n^* \otimes \dgr_N\la n\ra$.

Suppose $j \le N-1$.  Then $\Ext^k(C, \Lgr_j\la m\ra) = 0$ for all $k \ge 0$ and $m \in \Z$, so
\begin{equation}\label{eqn:qher-ext-iso}
\Ext^k(\wDelta^\gr_i, \Lgr_j\la m\ra) \cong \Ext^k(\wDelta^{\gr\prime}_i, \Lgr_j\la m\ra)
\qquad\text{for all $k \ge 0$, if $j \le N-1$.}
\end{equation}
In particular, we have
\begin{equation}
\label{eqn:qher-head1}
\dim \Hom(\wDelta^\gr_i, \Lgr_j\la m\ra) = 
\begin{cases}
1 & \text{if $j = i$ and $m = 0$,} \\
0 & \text{in all other cases with $j \le N-1$}
\end{cases}
\end{equation}
and, using induction and Lemma~\ref{lem:qher-closed-subset},
\begin{equation}
\label{eqn:qher-ext1}
\Ext^k(\wDelta^\gr_i, \Lgr_j\la m\ra) = 0 \quad\text{for $k \ge 1$, if
$\begin{cases}
\text{$j \le N-1$ and $m \ne 0$, or} \\
\text{$j \le i$ and $m = 0$.} 
\end{cases}$}
\end{equation}

Next, let $K$ be the kernel of the map $\dgr_N \to \Lgr_N$.  Note that if $m < 0$, then every composition factor of $K\la m\ra$ is isomorphic to some $\Lgr_j\la n\ra$ with $n < 0$ and $j \leq N-1$.  Assume $m < 0$, and consider the following long exact sequences:
{\small\[
\xymatrix@R=15pt@C=10pt{
\Hom(\wDelta^{\gr\prime}_i, K\la m\ra) \ar[r]\ar[d] &
  \Hom(\wDelta^{\gr\prime}_i, \dgr_N\la m\ra) \ar[r]\ar[d] &
  \Hom(\wDelta^{\gr\prime}_i, \Lgr_N\la m\ra) \ar[r]\ar[d] &
  \Ext^1(\wDelta^{\gr\prime}_i, K\la m\ra) \ar[d] \\
\Hom(\wDelta^\gr_i, K\la m\ra) \ar[r] &
  \Hom(\wDelta^\gr_i, \dgr_N\la m\ra) \ar[r] &
  \Hom(\wDelta^\gr_i, \Lgr_N\la m\ra) \ar[r] &
  \Ext^1(\wDelta^\gr_i, K\la m\ra).}
\]}%
Since $\Hom(C,K\la m\ra) = 0$, the first vertical map is an isomorphism. By~\eqref{eqn:qher-ext-iso} and~\eqref{eqn:qher-ext1}, both groups in the last column vanish.  It follows from~\eqref{eqn:qher-can} that the second vertical map is an isomorphism.  Therefore, by the five lemma, the third one is also an isomorphism, and we have $\Hom(\wDelta^\gr_i, \Lgr_N\la m\ra) = 0$ for $m < 0$.  In fact, we have
\begin{equation}\label{eqn:qher-head2}
\Hom(\wDelta^\gr_i, \Lgr_N\la m\ra) = 0 \qquad\text{for all $m \in \Z$.}
\end{equation}
For $m \ge 0$, this follows from~\eqref{eqn:wdelta-defn}, since $\Hom(C,\Lgr_N\la m\ra) = 0$ for $m \ge 0$.

Next, for $m < 0$, consider the exact sequence
\[
\cdots \to \Ext^1(\wDelta^\gr_i, \dgr_N\la m\ra) \to \Ext^1(\wDelta^\gr_i, \Lgr_N\la m\ra) \to \Ext^2(\wDelta^\gr_i, K\la m\ra) \to \cdots.
\]
The isomorphism~\eqref{eqn:qher-can} implies that the first term vanishes.  On the other hand, the last term vanishes by~\eqref{eqn:qher-ext1}.  We conclude that 
\begin{equation}\label{eqn:qher-ext2}
\Ext^1(\wDelta^\gr_i, \Lgr_N\la m\ra) = 0 \qquad \text{for $m < 0$.}
\end{equation}

Finally, let $M$ be the cokernel of $\Lgr_N \hookrightarrow \ngr_N$.  We will study $\Ext$-groups involving $M\la m\ra$ with $m < 0$.  Let $M' = W_{-m-1}M$ and $M'' = M/W_{-m-1}M$.  In other words, $M'\la m\ra = W_{-1}(M\la m\ra)$ and $M''\la m\ra = (M\la m\ra)/W_{-1}(M\la m\ra)$. All composition factors of $M'\la m\ra$ are of the form $\Lgr_j \la n\ra$ with $n < 0$ and $j \le N-1$, so by~\eqref{eqn:qher-ext1}, we have
\begin{equation}\label{eqn:qher-ext-ann1}
\Ext^k(\wDelta^\gr_i, M'\la m\ra) = 0 \qquad\text{for all $k \ge 1$.}
\end{equation}
We also have a short exact sequence $0 \to \Gr^W_{-m} M \to M'' \to M/W_{-m}M \to 0$.  It follows from~\eqref{eqn:pos-higher-ext} that $\Ext^k(\wDelta^\gr_i, (M/W_{-m}M)\la m\ra) = 0$ for all $k \ge 0$, so $\Ext^k(\wDelta^\gr_i, M''\la m\ra) \cong \Ext^k(\wDelta^\gr_i, \Gr^W_{-m}M\la m\ra)$.  Another invocation of~\eqref{eqn:pos-higher-ext} shows that $\Ext^k(\wDelta^\gr_i, \Gr^W_{-m}M\la m\ra) \cong \Ext^k(\Gr^W_0 \wDelta^\gr_i, \Gr^W_{-m}M\la m\ra)$.  Now, by construction, $\Gr^W_0 \wDelta^\gr_i \cong \Gr^W_0 \dgr_i = \Delta^\circ_i$.  On the other hand, since $-m > 0$, $\Gr^W_{-m} M\la m\ra \cong \Gr^W_{-m} \ngr_N\la m\ra$.  The latter object has a costandard filtration as an object of $\cA^\circ$, since $\cA$ is metameric by assumption.  By~\eqref{eqn:qher-pos-ff}, we have that $\Ext^k(\Delta^\circ_i, \Gr^W_{-m}M\la m\ra) = 0$ for $k \ge 1$. Unwinding the last few sentences, we find that $\Ext^k(\wDelta^\gr_i, M''\la m\ra) = 0$ for all $k \ge 1$.  Combining this with~\eqref{eqn:qher-ext-ann1} yields
\[
\Ext^k(\wDelta^\gr_i, M\la m\ra) = 0 \qquad\text{for all $k \ge 1$.}
\]
As a consequence, the natural map $\Ext^k(\wDelta^\gr_i, \Lgr_N) \to \Ext^k(\wDelta^\gr_i, \ngr_N)$  is an isomorphism for $k \ge 2$.  The latter group vanishes because $\wDelta^\gr_i$ has a standard filtration.  We conclude that
\begin{equation}\label{eqn:qher-ext3}
\Ext^k(\wDelta^\gr_i, \Lgr_N\la m\ra) = 0 \qquad \text{for $m<0$ and $k \ge 2$.}
\end{equation}

Both~\eqref{eqn:qher-ext2} and~\eqref{eqn:qher-ext3} were proved above for $m < 0$.  But they both hold for $m > 0$ as well: this follows immediately from~\eqref{eqn:pos-higher-ext} because every composition factor of $\wDelta^\gr_i$ is, by construction, of the form $\Lgr_u\la n\ra$ with $n \le 0$.

We have finished the study of $\wDelta^\gr_i$.  To summarize, property~\eqref{it:qher-filt} in the theorem holds by construction, and property~\eqref{it:qher-head} holds by~\eqref{eqn:qher-head1} and~\eqref{eqn:qher-head2}.  Property~\eqref{it:qher-ext0} is covered by~\eqref{eqn:qher-ext1}, and property~\eqref{it:qher-ext-neg} is obtained by combining~\eqref{eqn:qher-ext1}, \eqref{eqn:qher-ext2}, and~\eqref{eqn:qher-ext3}.

The construction of $\wnabla^\gr_i$ is similar and will be omitted.

We now turn to the last assertion in the theorem.  Assume that $\cA$ contains a family of objects $\{ \wDelta^\gr_s, s \in \scS \}$ satisfying properties~\eqref{it:qher-head}--\eqref{it:qher-filt}.  Let $s, t \in \scS$, and let $m < 0$.  A routine argument with weight filtrations, using~\eqref{eqn:pos-higher-ext} and property~\eqref{it:qher-ext-neg} (similar to the discussion following~\eqref{eqn:qher-ext-ann1}) shows that
\[
\Ext^1(\wDelta^\gr_s, \ngr_t\la m\ra) \cong \Ext^1 \bigl( \Gr^W_0 \wDelta^\gr_s, (\Gr^W_{-m} \ngr_t) \la m\ra \bigr).
\]
The left-hand side vanishes because $\wDelta^\gr_s$ has a standard filtration.  On the other hand, it follows from property~\eqref{it:qher-filt} that $\Gr^W_0 \wDelta^\gr_s \cong \Gr^W_0 \dgr_s \cong \Delta^\circ_s$. Therefore, $\Ext^1 \bigl( \Delta^\circ_s, (\Gr^W_{-m} \ngr_t) \la m\ra \bigr)=0$.  We have computed this $\Ext^1$-group in $\cA$, but its vanishing implies that
\[
\Ext^1_{\cA^\circ} \bigl( \Delta^\circ_s, (\Gr^W_{-m} \ngr_t)\la m\ra \bigr) = 0 \qquad\text{for all $s \in \scS$}
\]
as well.  By a standard argument (see e.g.~\cite[Proposition~A2.2(iii)]{donkin}), we conclude that $(\Gr^W_{-m} \ngr_t)\la m\ra$ has a costandard filtration for all $m < 0$. A dual argument shows that each $(\Gr^W_m \dgr_t)\la -m\ra$ has a standard filtration, so $\cA$ is metameric.
\end{proof}

\begin{rmk}
In a metameric category, the description of projectives from~\cite[Theorem~3.2.1]{bgs} or~\cite[Theorem~A.3]{modrap2} can be refined somewhat, as follows.  Let $\cA$ be a metameric category, and let $P^\gr_s$ and $P^\circ_s$ be projective covers of $\Lgr_s$ in $\cA$ and in $\cA^\circ$, respectively.  Then $P^\gr_s$ admits a filtration whose subquotients are various $\wDelta^\gr_t$ with $t \ge s$.  Moreover, we have
\[
(P^\gr_s : \wDelta^\gr_t) = (P^\circ_s : \Delta^\circ_t).
\]
The proof is a straightforward generalization of that of~\cite[Theorem~3.2.1]{bgs}. As this result will not be needed in this paper, we omit the details.
\end{rmk}

\subsection{Koszul and standard Koszul categories}
\label{ss:koszul}

Recall that a graded quasihereditary category is said to be \emph{Koszul} if it satisfies
\begin{equation}
\label{eqn:def-koszul}
\Ext^k(\Lgr_s, \Lgr_t\la n\ra) = 0 \qquad\text{unless $n = -k$.}
\end{equation}
(A Koszul category is automatically positively graded by Proposition~\ref{prop:pos-qher}.)  
It is said to be \emph{standard Koszul} if it satisfies
\begin{equation}
\label{eqn:def-standard-koszul}
\Ext^k(\Lgr_s, \ngr_t\la n\ra) = \Ext^k(\dgr_s, \Lgr_t\la n\ra) = 0
\qquad\text{unless $n = -k$.}
\end{equation}
(See~\cite{adl, mazorchuk} for this notion; see also~\cite{irving} for an earlier study of this condition.) 

The following well-known result follows from~\cite{adl}. Since the latter paper uses a vocabulary which is is quite different from ours, we include a proof.

\begin{prop}
\label{prop:standard-koszul}
Let $\cA$ be 
a graded quasihereditary category. If $\cA$ is standard Koszul, then it is Koszul.
\end{prop}

\begin{proof}
We prove the result by induction on the cardinality of $\scS$. The claim is obvious if $\scS$ consists of only one element, since in this case $\cA$ is semisimple.

Now assume that $\scS$ has at least two elements, and that $\cA$ is standard Koszul. Let $s \in \scS$ be minimal, and set $\cA':=\cA_{\{s\}}$, $\cA'':=\cA/\cA_{\{s\}}$, $\iota:=\iota_{\{s\}}$, $\Pi:=\Pi_{\{s\}}$. By Lemma~\ref{lem:qher-closed-subset}, these categories are graded quasihereditary. We claim that $\cA''$ is standard Koszul. Indeed for $t,u \in \scS \smallsetminus \{s\}$ we have
\[
\Ext^k_{\cA''}(\Pi(\Lgr_t), \Pi(\ngr_u) \la n \ra) \cong \Ext^k_{\cA}(\Lgr_t, \Pi^R \circ \Pi(\ngr_u) \la n \ra) \cong \Ext^k_{\cA}(\Lgr_t, \ngr_u \la n \ra)
\]
by~\eqref{eqn:adjoints-pi-d-n}, and the right-hand side vanishes unless $n=-k$ by assumption. Similarly we have
\[
\Ext^k_{\cA''}(\Pi(\dgr_t), \Pi(\Lgr_u) \la n \ra) \cong \Ext^k_{\cA}(\dgr_t, \Lgr_u \la n \ra),
\]
and again the right-hand side vanishes unless $n=-k$.

By induction, we deduce that $\cA''$ is Koszul. Now let $t \in \scS$, and consider the distinguished triangle
\[
\iota \circ \iota^R (\Lgr_t) \to \Lgr_t \to \Pi^R \circ \Pi (\Lgr_t) \xrightarrow{[1]}
\]
of Lemma~\ref{lem:qher-closed-subset}.
Applying the functor $\Hom(\Lgr_u, -\la n \ra)$ (for some $u \in \scS$ and $n \in \Z$) we obtain a long exact sequence
\begin{multline*}
\cdots \to \Ext^k_{\cA'}(\iota^L(\Lgr_u), \iota^R(\Lgr_t) \la n \ra) \to \Ext^k_{\cA}(\Lgr_u, \Lgr_t \la n \ra) \\
\to \Ext^k_{\cA''}(\Pi(\Lgr_u), \Pi(\Lgr_t) \la n \ra) \to \cdots
\end{multline*}
Since $\cA''$ is Koszul, the third term vanishes unless $n=-k$. Hence to conclude it suffices to prove that the first term also vanishes unless $n=-k$.

We claim that $\iota^L(\Lgr_u)$ is a direct sum of objects of the form $\Lgr_s \la m \ra [-m]$ for some $m \in \Z$. Indeed, since $\cA'$ is semisimple, this object is a direct sum of objects of the form $\Lgr_s \la a \ra [b]$. But if such an object appears as a direct summand then
\[
\Hom_{\Db(\cA')}(\iota^L(\Lgr_u), \Lgr_s \la a \ra [b]) \cong \Ext_{\cA}^b(\Lgr_u, \Lgr_s \la a \ra) \neq 0,
\]
which implies that $a=-b$ since $\Lgr_s=\ngr_s$; this finishes the proof of the claim. Similar arguments show that $\iota^R(\Lgr_t)$ is also a direct sum of objects of the form $\Lgr_s \la m' \ra [-m']$ for $m' \in \Z$. One deduces that indeed $\Ext^k_{\cA'}(\iota^L(\Lgr_u), \iota^R(\Lgr_t) \la n \ra)=0$ unless $n=-k$, which finishes the proof.
\end{proof}

\begin{rmk}
The proof shows that the condition that $\Ext^k(\Lgr_s, \ngr_t\la n\ra)=0$ unless $k=-n$ already implies that $\cA$ is Koszul.  Similar arguments using the functors $\iota^L$, $\Pi^L$ instead of $\iota^R$, $\Pi^R$ show that if $\cA$ satisfies the dual condition that $\Ext^k(\dgr_s, \Lgr_t\la n\ra)=0$ unless $k=-n$, then $\cA$ is Koszul.
\end{rmk}

\subsection{$Q$-Koszul and standard $Q$-Koszul categories}
\label{ss:qkoszul}

In this subsection we study a generalization of the notions considered in~\S\ref{ss:koszul} that has been recently introduced by Parshall--Scott~\cite[\S3]{ps}. 

\begin{defn}\label{defn:q-koszul}
Let $\cA$ be a positively graded quasihereditary category. It is said to be \emph{$Q$-Koszul} if
\[
\Ext^k_\cA(\Delta^\circ_s, \nabla^\circ_t\la n\ra) = 0
\qquad\text{unless $n = -k$.}
\]
It is said to be \emph{standard $Q$-Koszul} if
\[
\Ext^k_\cA(\Delta^\circ_s, \ngr_t\la n\ra) = \Ext^k_\cA(\dgr_s, \nabla^\circ_t\la n\ra) = 0
\qquad\text{unless $n = -k$.}
\]
\end{defn}

The following result is an analogue of Proposition~\ref{prop:standard-koszul} in this context.  The same result appears in~\cite[Corollary~3.2]{ps:qka}, but in a somewhat different language, so as with Proposition~\ref{prop:standard-koszul}, we include a proof.

\begin{prop}
\label{prop:standard-Q-Koszul}
Let $\cA$ be 
a positively graded quasihereditary category. If $\cA$ is standard $Q$-Koszul, then it is $Q$-Koszul.
\end{prop}

\begin{proof}
The proof is very similar to that of Proposition~\ref{prop:standard-koszul}. We proceed by induction on the cardinality of $\scS$, the base case being obvious. We choose $s \in \scS$ minimal, and set $\cA':=\cA_{\{s\}}$, $\cA'':=\cA/\cA'$, $\iota:=\iota_{\{s\}}$, $\Pi:=\Pi_{\{s\}}$. By Corollary~\ref{cor:quotient-pos}, the category $\cA''$ is positively graded. It is also clear that $\Gr^W_0(\Pi(\dgr_t)) \cong \Pi(\Delta^\circ_t)$ and $\Gr^W_0(\Pi(\ngr_t)) \cong \Pi(\nabla^\circ_t)$ for $t \neq s$. Then using~\eqref{eqn:adjoints-pi-d-n} as in the proof of Proposition~\ref{prop:standard-koszul}, one obtains that $\cA''$ is standard $Q$-Koszul; hence by induction it is $Q$-Koszul.

Now consider, for $t \in \scS$, the distinguished triangle
\[
\iota \circ \iota^R (\nabla^\circ_t) \to \nabla^\circ_t \to \Pi^R \circ \Pi (\nabla^\circ_t) \xrightarrow{[1]}.
\]
Applying the functor $\Hom(\Delta^\circ_u, -\la n \ra)$ (for some $u \in \scS$ and $n \in \Z$) we obtain a long exact sequence
\begin{multline*}
\cdots \to \Ext^k_{\cA'}(\iota^L(\Delta^\circ_u), \iota^R(\nabla^\circ_t) \la n \ra) \to \Ext^k_{\cA}(\Delta^\circ_u, \nabla^\circ_t \la n \ra) \\
\to \Ext^k_{\cA''}(\Pi(\Delta^\circ_u), \Pi(\nabla^\circ_t) \la n \ra) \to \cdots
\end{multline*}
By induction, the third term vanishes unless $n=-k$. Now one can easily check that both $\iota^L(\Delta^\circ_u)$ and $\iota^R(\nabla^\circ_t)$ are direct sums of objects of the form $\Lgr_s \la m \ra [-m]$ for $m \in \Z$, and we deduce that the first term also vanishes unless $n=-k$, which finishes the proof.
\end{proof}

\begin{rmk}\label{rmk:t-koszul}
It is natural to ask whether there is a notion of ``Koszul duality'' for $Q$-Koszul categories.

Recall that classical Koszul duality is a kind of derived equivalence that sends simple objects in one category to projective objects in the other.  There is a generalization of this notion due to Madsen~\cite{madsen}.  Suppose $\cA$ is a finite-length (but not necessarily quasihereditary) category satisfying conditions~\eqref{it:pos-ext} or~\eqref{it:pos-wt-filt} of Proposition~\ref{prop:pos-qher}.  Then it still makes sense to define a Serre subcategory $\cA^\circ$ as in Proposition~\ref{prop:parshall-scott}.  Assume that $\cA^\circ$ has the structure of a quasihereditary category, and that for any two tilting objects $T^\circ_s, T^\circ_t \in \cA^\circ$, we have
\[
\Ext^k_\cA(T^\circ_s, T^\circ_t\la n\ra) = 0\qquad\text{unless $n = -k$.}
\]
Such a category $\cA$ is said to be \emph{$T$-Koszul}.  Madsen's theory leads to a new $T$-Koszul abelian category $\mathcal{B}$ and a derived equivalence $\Db(\cA) \simto \Db(\mathcal{B})$ such that tilting objects of $\cA^\circ$ correspond to projective objects in $\mathcal{B}$.  If $\cA^\circ$ happens to be semisimple, then Madsen's notion reduces to ordinary Koszul duality.

Clearly, every $Q$-Koszul category is $T$-Koszul. But it is not known whether the $T$-Koszul dual of a $Q$-Koszul category must be $Q$-Koszul, see~\cite[Questions~4.2]{ps:qka}. 
\end{rmk}

\section{Weights}
\label{sec:weights}

\subsection{Setting}
\label{ss:setting}

In this section (and the next one) we work in the setting of~\cite[\S\S2--3]{modrap2}. In particular, we choose a prime number $\ell$ and a finite extension $\K$ of $\Ql$.  We denote by $\O$ the ring of integers of $\K$ and by $\F$ the residue field of $\O$. We use the letter $\E$ to denote any member of $(\K,\O,\F)$.

We fix a complex algebraic variety $X$ endowed with a finite stratification $X=\bigsqcup_{s \in \scS} X_s$ where each $X_s$ is isomorphic to an affine space. We denote by $\Db_\scS(X,\E)$ the derived $\scS$-constructible category of sheaves on $X$, with coefficients in $\E$. The cohomological shift in this category will be denoted $\{1\}$.
We assume that the assumptions $\textbf{(A1)}$ (``existence of enough parity complexes'') and $\textbf{(A2)}$ (``standard and costandard objects are perverse'') of~\cite{modrap2} are satsified. Then one can consider the additive category $\Parity_{\scS}(X,\E)$ of parity complexes on $X$ (in the sense of~\cite{jmw}; see~\cite[\S 2.1]{modrap2} for a reminder of the main properties of this category) and the ``mixed derived category'' $\Dmix_{\scS}(X,\E):=\Kb \Parity_{\scS}(X,\E)$. This category possesses two important autoequivalences: the cohomological shift $[1]$, and the ``internal'' shift $\{1\}$ induced by the shift functor on parity complexes. We also set $\la 1 \ra := \{-1\}[1]$. If $h : Y \to X$ is a locally closed inclusion of a union of strata, then we have well-defined functors
\[
h_*, h_! : \Dmix_{\scS}(Y,\E) \to \Dmix_{\scS}(X,\E), \qquad h^*, h^! : \Dmix_{\scS}(Y,\E) \to \Dmix_{\scS}(X,\E)
\]
which satisfy all the usual properties; see~\cite[\S 2.5]{modrap2}. (Here and below, we also denote by $\scS$ the restriction of the stratification to $Y$.) We also have ``extension of scalars'' functors
\[
\K : \Dmix_{\scS}(X,\O) \to \Dmix_{\scS}(X,\K), \qquad \F : \Dmix_{\scS}(X,\O) \to \Dmix_{\scS}(X,\F)
\]
and a ``Verdier duality'' antiequivalence
\[
\D : \Dmix_{\scS}(X,\E) \simto \Dmix_{\scS}(X,\E).
\]

The triangulated category $\Dmix_{\scS}(X,\E)$ can be endowed with a ``perverse t-struc\-ture''; see~\cite[Definition~3.3]{modrap2}. We denote by $\Perv^\mix_\scS(X,\E)$ the heart of this t-structure. Objects of $\Perv^\mix_\scS(X,\E)$ will be called ``mixed perverse sheaves.'' If $\E=\O$ or $\K$, this category is a graded quasihereditary category, with shift functor $\la 1 \ra$, simple objects $\IC^\mix_s:=i_{s!*} \uuE_{X_s}$, standard objects $\dmix_s := i_{s!} \uuE_{X_s}$, and costandard objects $\nmix_s:=i_{s*} \uuE_{X_s}$. (Here, $i_s : X_s \to X$ is the inclusion, and $\uuE_{X_s} := \underline{\E}_{X_s} \{\dim(X_s)\}$, where $\underline{\E}_{X_s}$ is the constant sheaf on $X_s$, an object of $\Parity_{\scS}(X_s,\E)$.) We denote by $\cP^\mix_s$ the projective cover of $\IC_s^\mix$, and by $\cT^\mix_s$ the indecomposable tilting object associated with $s$. When necessary, we add a mention of the coefficients ``$\E$'' we consider. Note in particular that we have
\[
\K(\IC_s^\mix(\O)) \cong \IC_s^\mix(\K), \quad \F(\cP_s^\mix(\O)) \cong \cP_s^\mix(\F), \quad \F(\cT_s^\mix(\O)) \cong \cT_s^\mix(\F).
\]
(For all of this, see~\cite[\S\S 3.2--3.3]{modrap2}.)

As in~\cite{modrap2}, we denote by $\cE_s$ the unique indecomposable parity complex which is supported on $\overline{X_s}$ and whose restriction to $X_s$ is $\uuE_{X_s}$. We denote this same object by $\cE^\mix_s$ when it is regarded as an object of $\Dmix_{\scS}(X,\E)$.

We do not know whether $\Perv^\mix_\scS(X,\F)$ is positively graded under these assumptions.  The main result of this section, Proposition~\ref{prop:perv-pos}, gives a number of conditions that are equivalent to $\Perv^\mix_\scS(X,\F)$ being positively graded.  Along the way to that result, we construct a candidate abelian category $\Perv^\circ_\scS(X,\F)$ that ``should'' be the category $\cA^\circ$ of Proposition~\ref{prop:parshall-scott} in this case.  However, $\Perv^\circ_\scS(X,\F)$ is defined even when $\Perv^\mix_\scS(X,\F)$ is not positively graded. 

\subsection{Weights}
\label{ss:weights}

We begin by introducing a notion that will ``morally'' play the same role in $\Dmix_\scS(X,\E)$ that is played by Deligne's theory of weights (see~\cite[\S5]{bbd}) in the realm of $\ell$-adic \'etale sheaves.

\begin{defn}
An object $M \in \Dmix_\scS(X,\E)$ is said to \emph{have weights${}\le n$} (resp.~${}\ge n$) if it is isomorphic to a complex $\cdots \to M^{-1} \to M^0 \to M^1 \to \cdots$ of objects in $\Parity_{\scS}(X,\E)$ in which $M^i = 0$ for all $i < -n$ (resp.~$i > -n$).  It is said to be \emph{pure of weight $n$} if has weights${}\le n$ and${}\ge n$.
\end{defn}

The full subcategory of $\Dmix_\scS(X,\E)$ consisting of objects with weights${}\le n$ (resp.~${}\ge n$) is denoted $\Dmix_\scS(X,\E)_{\le n}$ (resp.~$\Dmix_\scS(X,\E)_{\ge n}$).  Using standard arguments in triangulated categories one can check that
these categories admit the following alternative characterizations:
\begin{align*}
\Dmix_\scS(X,\E)_{\le n} &= \{ M \mid \text{$\Hom(M, \cE^\mix_s\{m\}[k]) = 0$ for all $m \in \Z$ and all $k > n$} \}, \\
\Dmix_\scS(X,\E)_{\ge n} &= \{ M \mid \text{$\Hom(\cE^\mix_s\{m\}[k], M) = 0$ for all $m \in \Z$ and all $k < n$} \},
\end{align*}
and moreover that an object in $\Dmix_\scS(X,\E)$ is pure of weight $n$ if and only if it is a direct sum of objects of the form $\cE^\mix_s\{m\}[n]$.

Note that weights are stable under extensions.  That is, if the first and third terms of a distinguished triangle have weights${}\le n$ (resp.~${}\ge n$), then the same holds for the middle term.

\begin{ex}\label{ex:wts-stratum}
Consider a single stratum $X_s$.  For a finitely-generated $\E$-module $N$, let $\underline{N}$ denote the corresponding constant sheaf on $X_s$, and let $\uuline{N} = \uline{N}\{\dim X_s\}$. (Here we use the same convention as in~\cite[\S 3.5]{modrap2} in case $\E=\O$ and $N$ is not free.) Every object $M \in \Dmix_\scS(X_s,\E)$ is isomorphic (canonically if $\E = \F$ or $\K$, and noncanonically if $\E = \O$) to a finite direct sum
\begin{equation}\label{eqn:wts-stratum}
M \cong \bigoplus_{i, j \in \Z} \uuline{M}{}^i_j\{j\}[-i]
\end{equation}
where the $M^i_j$ are various finitely generated $\E$-modules.  With this in mind:
\begin{enumerate}
\item If $\E = \F$ or $\K$, then $M$ has weights${}\le 0$ if and only if $M^i_j = 0$ for all $i < 0$.
\item If $\E = \O$, then $M$ has weights${}\le 0$ if and only if $M^i_j = 0$ for all $i < 0$, and all $M^0_j$ are torsion-free.
\end{enumerate}
\end{ex}

\begin{lem}
\label{lem:weights-dmix-nmix}
For any $s \in \scS$, $\dmix_s$ has weights${}\le 0$, and $\nmix_s$ has weights${}\ge 0$.
\end{lem}
\begin{proof}
It is clear by adjunction (and using~\cite[Remark~2.7]{modrap2}) that we have $\Hom(\dmix_s, \cE^\mix_t\{m\}[k]) = 0$ if $k \neq 0$, and similarly for $\nmix_s$.
\end{proof}

\begin{lem}\label{lem:wt-stability}
Let $j: U \hookrightarrow X$ be the inclusion of an open union of strata, and let $i: Z \hookrightarrow X$ be the complementary closed inclusion.
\begin{enumerate}
\item $j^{*}$ and $i_{*}$ preserve weights.\label{it:wt-stab-exact}
\item $j_{!}$ sends $\Dmix_\scS(U,\E)_{\le n}$ to $\Dmix_\scS(X,\E)_{\le n}$, and $j_{*}$ sends $\Dmix_\scS(U,\E)_{\ge n}$ to $\Dmix_\scS(X,\E)_{\ge n}$.\label{it:wt-stab-j}
\item $i^{*}$ sends $\Dmix_\scS(X,\E)_{\le n}$ to $\Dmix_\scS(Z,\E)_{\le n}$, and $i^{!}$ sends $\Dmix_\scS(X,\E)_{\ge n}$ to $\Dmix_\scS(Z,\E)_{\ge n}$.\label{it:wt-stab-i}
\item If $Z$ consists of a single stratum, then $i^{*}$ and $i^{!}$ preserve weights.\label{it:wt-stab-strat}
\item $\D$ exchanges $\Dmix_\scS(X,\E)_{\le n}$ and $\Dmix_\scS(X,\E)_{\ge -n}$.\label{it:wt-stab-d}
\end{enumerate}
\end{lem}
\begin{proof}
Parts~\eqref{it:wt-stab-exact}, \eqref{it:wt-stab-strat}, and~\eqref{it:wt-stab-d} are clear, because in those cases, the functors take parity complexes to parity complexes.  Parts~\eqref{it:wt-stab-j} and~\eqref{it:wt-stab-i} then follow from part~\eqref{it:wt-stab-exact} by adjunction.
\end{proof}

\begin{lem}\label{lem:wt-stalks}
Let $\cF \in \Dmix_\scS(X,\E)$.  We have:
\begin{enumerate}
\item $\cF$ has weights${}\le n$ if and only if $i_s^*\cF$ has weights${}\le n$ for all $s \in \scS$.
\item $\cF$ has weights${}\ge n$ if and only if $i_s^!\cF$ has weights${}\ge n$ for all $s \in \scS$.
\end{enumerate}
\end{lem}
\begin{proof}
We will only treat the first assertion.  The ``only if'' direction is part of Lemma~\ref{lem:wt-stability}, so we need only prove the ``if'' direction.  In that case, we proceed by induction on the number of strata in $X$.  If $X$ consists of a single stratum, there is nothing to prove.  Otherwise, suppose $i_s^*\cF$ has weights${}\le n$ for all $s$.  Choose a closed stratum $X_s \subset X$.  Let $U = X \smallsetminus X_s$, and let $j: U \hookrightarrow X$ be the inclusion map.  Then $j^*\cF$ has weights${}\le n$ by induction.  The first and last terms of the distinguished triangle $j_!j^*\cF \to \cF \to i_{s*}i_s^*\cF \to$ have weights${}\le n$ by Lemma~\ref{lem:wt-stability}, so the middle term does as well.
\end{proof}

\subsection{Baric truncation functors}
\label{ss:baric}

For $n \in \Z$, we define two full triangulated subcategories of $\Dmix_\scS(X,\E)$ as follows:
\begin{equation}\label{eqn:baric-defn}
\begin{aligned}
\Dmix_\scS(X,\E)_{\bleq n} &:= \text{the subcategory generated by the $\dmix_s\la m\ra$ for $m \le n$}\\
\Dmix_\scS(X,\E)_{\bgeq n} &:= \text{the subcategory generated by the $\nmix_s\la m\ra$ for $m \ge n$}
\end{aligned}
\end{equation}
We also put
\[
\Dmix_\scS(X,\E)^\circ := \Dmix_\scS(X,\E)_{\bleq 0} \cap \Dmix_\scS(X,\E)_{\bgeq 0}.
\]

\begin{ex}\label{ex:baric-stratum}
With the notation of Example~\ref{ex:wts-stratum}, the object $M$ in~\eqref{eqn:wts-stratum} lies in $\Dmix_\scS(X_s,\E)_{\bleq 0}$ if and only if $M^i_j = 0$ for all $j < 0$.
\end{ex}

\begin{lem}\label{lem:baric-basic}
\begin{enumerate}
\item 
For any $A \in \Dmix_\scS(X,\E)_{\bleq n}$ and $B \in \Dmix_\scS(X,\E)_{\bgeq n+1}$, we have  $\Hom(A,B) = 0$.\label{it:baric-orth}
\item 
The inclusion $\Dmix_\scS(X,\E)_{\bleq n} \hookrightarrow \Dmix_\scS(X,\E)$ admits a right adjoint $\beta_{\bleq n}: \Dmix_\scS(X,\E) \to \Dmix_\scS(X,\E)_{\bleq n}$, which is a triangulated functor.
Similarly, the inclusion $\Dmix_\scS(X,\E)_{\bgeq n} \hookrightarrow \Dmix_\scS(X,\E)$ admits a left adjoint $\beta_{\bgeq n}: \Dmix_\scS(X,\E) \to \Dmix_\scS(X,\E)_{\bgeq n}$, which is a triangulated functor.\label{it:baric-trunc}
\item For every object $M \in \Dmix_\scS(X,\E)$ and every $n \in \Z$, there is a functorial distinguished triangle
\[
\beta_{\bleq n}M \to M \to \beta_{\bgeq n+1}M \xrightarrow{[1]}.
\]
Moreover, if $M' \in \Dmix_\scS(X,\E)_{\bleq n}$ and $M'' \in \Dmix_\scS(X,\E)_{\bgeq n+1}$, for any distinguished triangle $M' \to M \to M'' \xrightarrow{[1]}$ there exist canonical isomorphisms $\varphi : M' \simto \beta_{\bleq n}M$ and $\psi : M'' \simto \beta_{\bgeq n+1}M$ such that $(\varphi, \id_M, \psi)$ is an isomorphism of distinguished triangles.\label{it:baric-tri}
\item All the $\beta_{\bleq n}$ and $\beta_{\bgeq m}$ commute with one another.\label{it:baric-comm}
\end{enumerate}
\end{lem}
\begin{proof}
Part~\eqref{it:baric-orth} is immediate from the definitions and~\cite[Lemma~3.2]{modrap2}.  Next, we prove a weak version of part~\eqref{it:baric-tri}.  It is clear that the collection of objects
\[
C = \{ \dmix_s\la m\ra \mid m \le 0 \} \cup \{ \nmix_s\la m\ra \mid m \ge 1 \}
\]
generates $\Dmix_\scS(X,\E)$ as a triangulated category.  Let us express this another way, using the ``$\ast$'' notation of~\cite[\S1.3.9]{bbd}: for any object $M \in \Dmix_\scS(X,\E)$, there are objects $C_1, \ldots, C_n \in C$ and integers $k_1, \ldots, k_n$ such that
\begin{equation}\label{eqn:baric-trunc}
M \in C_1[k_1] \ast \cdots \ast C_n[k_n].
\end{equation}
Now, observe that if $a \le 0$ and $b \ge 1$, then $\nmix_s\la b\ra[p] \ast \dmix_t\la a\ra[q]$ contains only the object $\nmix_s\la b\ra[p] \oplus \dmix_t\la a\ra[q]$, because $\Hom(\dmix_t\la a\ra[q], \nmix_s\la b\ra[p+1]) = 0$.  Thus, 
\[
\nmix_s\la b\ra[p] \ast \dmix_t\la a\ra[q] \subset \dmix_t\la a\ra[q] \ast \nmix_s\la b\ra[p].
\]
Using this fact, we can rearrange the expression~\eqref{eqn:baric-trunc} so that the following holds: there is some $n' \leq n$ such that $C_1, \ldots, C_{n'}$ are all of the form $\dmix_s\la m\ra$ with $m \le 0$, while $C_{n'+1}, \ldots, C_n$ are of the form $\nmix_s\la m\ra$ with $m \ge 1$.  Then~\eqref{eqn:baric-trunc} says that there is a distinguished triangle
\[
A \to M \to B \to
\]
where $A \in C_1 \ast \cdots \ast C_{n'} \subset \Dmix_\scS(X,\E)_{\bleq 0}$, and $B \in C_{n'+1} \ast \cdots \ast C_n \subset \Dmix_\scS(X,\E)_{\bgeq 1}$.  We have not yet proved that this triangle is functorial.  However, we have shown that the collection of categories $(\{ \Dmix_\scS(X,\E)_{\bleq n}\}, \{\Dmix_\scS(X,\E)_{\bgeq n}\})_{n \in \Z}$ satisfies the axioms of a so-called \emph{baric structure}~\cite[Definition~2.1]{at}.  The remaining statements in the lemma are general properties of baric structures from~\cite[Propositions~2.2 \&~2.3]{at}.
\end{proof}

\begin{rmk}
\label{rmk:baric-morphisms}
If $M$ is an object of $\Dmix_\scS(X,\E)$, then $M$ is in $\Dmix_\scS(X,\E)_{\bleq 0}$ iff $\Hom(M,\nmix_s \la m \ra[k])=0$ for all $s \in \scS$, $k \in \Z$ and $m \in \Z_{>0}$. Indeed, the ``only if'' part follows from Lemma~\ref{lem:baric-basic}\eqref{it:baric-orth}. To prove the ``if'' part, consider the baric truncation triangle $\beta_{\bleq 0}M \to M \to \beta_{\bgeq 1}M \xrightarrow{[1]}$ of Lemma~\ref{lem:baric-basic}\eqref{it:baric-tri}. Our assumption implies that the second arrow in this triangle is trivial, hence we deduce an isomorphism $\beta_{\bleq 0}M \cong M \oplus \beta_{\bgeq 1}M[-1]$. If $\beta_{\bgeq 1}M$ were non zero, then the projection $\beta_{\bleq 0}M \to \beta_{\bgeq 1}M[-1]$ would be non zero, contradicting Lemma~\ref{lem:baric-basic}\eqref{it:baric-orth}.
\end{rmk}

\begin{lem}\label{lem:baric-stability}
Let $j: U \hookrightarrow X$ be the inclusion of an open union of strata, and let $i: Z \hookrightarrow X$ be the complementary closed inclusion.
\begin{enumerate}
\item $j^{*}$ and $i_{*}$ commute with all $\beta_{\bleq n}$ and $\beta_{\bgeq n}$.\label{it:bar-stab-exact}
\item $j_{!}$ sends $\Dmix_\scS(U,\E)_{\bleq n}$ to $\Dmix_\scS(X,\E)_{\bleq n}$, and $j_{*}$ sends $\Dmix_\scS(U,\E)_{\bgeq n}$ to $\Dmix_\scS(X,\E)_{\bgeq n}$.\label{it:bar-stab-j}
\item $i^{*}$ sends $\Dmix_\scS(X,\E)_{\bleq n}$ to $\Dmix_\scS(Z,\E)_{\bleq n}$, and $i^{!}$ sends $\Dmix_\scS(X,\E)_{\bgeq n}$ to $\Dmix_\scS(Z,\E)_{\bgeq n}$.\label{it:bar-stab-i}
\item $\D$ exchanges $\Dmix_\scS(X,\E)_{\bleq n}$ and $\Dmix_\scS(X,\E)_{\bgeq -n}$.\label{it:bar-stab-d}
\end{enumerate}
\end{lem}
\begin{proof}
For the first three parts, it suffices to observe that $j^{*}$, $i_{*}$, $j_{!}$, and $i^{*}$ send standard objects to standard objects (or to zero), while $j^{*}$, $i_{*}$, $j_{*}$, and $i^{!}$ send costandard objects to costandard objects (or to zero).  Similarly, the last part follows from the fact that $\D$ exchanges standard and costandard objects.
\end{proof}

\begin{lem}
\label{lem:beta-K-F}
The functors $\beta_{\bleq n}$ and $\beta_{\bgeq n}$ commute with $\K({-})$ and $\F({-})$.
\end{lem}
\begin{proof}
Since extension of scalars sends standard objects to standard objects and costandard objects to costandard objects, it is clear that $\F({-})$ sends $\Dmix_\scS(X,\O)_{\bleq n}$ to $\Dmix_\scS(X,\F)_{\bleq n}$ and $\Dmix_\scS(X,\O)_{\bgeq n}$ to $\Dmix_\scS(X,\F)_{\bgeq n}$, and similarly for $\K({-})$. Then the result follows from Lemma~\ref{lem:baric-basic}\eqref{it:baric-tri}.
\end{proof}

\begin{lem}\label{lem:beta-stratum}
Suppose $X = X_s$ consists of a single stratum.  Then the functors $\beta_{\bleq n}$ and $\beta_{\bgeq n}$ are t-exact for the perverse t-structure on $\Dmix_\scS(X_s,\E)$.  In fact, for $M \in \Dmix_\scS(X_s,\E)$, there exists a canonical isomorphism $M \cong \beta_{\bleq n}M \oplus \beta_{\bgeq n+1}M$.
\end{lem}
\begin{proof}
Given $M \in \Dmix_\scS(X_s,\E)$, write a decomposition as in~\eqref{eqn:wts-stratum}, and form the distinguished triangle
\[
\bigoplus_{\substack{i \in \Z \\ j \ge -n}} \uuline{M}{}^i_j\{j\}[-i] \to M \to
\bigoplus_{\substack{i \in \Z \\ j \le -n - 1}} \uuline{M}{}^i_j\{j\}[-i] \xrightarrow{[1]}.
\]
Referring to Example~\ref{ex:baric-stratum}, we see that the first term belongs to $\Dmix_\scS(X_s,\E)_{\bleq n}$, and the third one to $\Dmix_\scS(X_s,\E)_{\bgeq n+1}$.  By Lemma~\ref{lem:baric-basic}\eqref{it:baric-tri}, this triangle must be canonically isomorphic to $\beta_{\bleq n}M \to M \to \beta_{\bgeq n+1} M \xrightarrow{[1]}$.  This triangle is clearly split.  Since $\Hom(\beta_{\bgeq n+1}M, \beta_{\bleq n}M)$ vanishes,
the splitting is canonical.  Finally, since any direct summand of a perverse sheaf is a perverse sheaf, the functors $\beta_{\bleq n}$ and $\beta_{\bgeq n}$ are t-exact.
\end{proof}

\subsection{A t-structure on $\Dmix_\scS(X,\E)^\circ$}
\label{ss:perv-circ}

In the following statement we use the notion of recollement from~\cite[\S 1.4]{bbd}.

\begin{prop}\label{prop:circ-recollement}
Let $j: U \hookrightarrow X$ be the inclusion of an open union of strata, and let $i: Z \hookrightarrow X$ be the complementary closed inclusion.  We have a recollement diagram
\[
\xymatrix@C=1.5cm{
\Dmix_\scS(Z,\E)^\circ \ar[r]^{i_*} &
\Dmix_\scS(X,\E)^\circ \ar[r]^{j^*} 
  \ar@/_1pc/[l]_{\beta_{\bgeq 0}i^{*}} \ar@/^1pc/[l]^{\beta_{\bleq 0}i^{!}} &
\Dmix_\scS(U,\E)^\circ. 
  \ar@/_1pc/[l]_{\beta_{\bgeq 0}j_{!}} \ar@/^1pc/[l]^{\beta_{\bleq 0}j_{*}}
}
\]
\end{prop}
\begin{proof}
The required adjunction properties for these functors, and the fact that $j^* i_*=0$, follow from the corresponding result for the mixed derived category; see~\cite[Proposition~2.3]{modrap2}.  Next, for $M \in \Dmix_\scS(Z,\E)^\circ$, consider the natural maps
\[
i^{*} i_*M \to (\beta_{\bgeq 0}i^{*}) i_*M \to M.
\]
It is easily checked that the composition is the morphism induced by adjunction, and so
is an isomorphism.  In particular, $i^{*}i_*M$ lies in $\Dmix_\scS(Z,\E)_{\bgeq 0}$, so the map $i^{*}i_*M \to \beta_{\bgeq 0}i^{*}i_*M$ is an isomorphism.  We conclude that the adjunction map $(\beta_{\bgeq 0}i^{*}) i_*M \to M$ is an isomorphism as well.  Similar arguments show that the adjunction morphisms $\id \to j^* (\beta_{\bgeq 0}j_{!})$, $\id \to (\beta_{\bleq 0}i^{!})i_*$, and $j^* (\beta_{\bleq 0}j_{*}) \to \id$ are isomorphisms.

Finally, given $M \in \Dmix_\scS(X,\E)^\circ$, form the triangle
$j_{!}j^*M \to M \to i_*i^{*}M \xrightarrow{[1]}$, and then apply $\beta_{\bgeq 0}$.  Using Lemma~\ref{lem:baric-stability}, we obtain a distinguished triangle
\[
(\beta_{\bgeq 0}j_{!}) j^* M \to M \to i_* (\beta_{\bgeq 0} i^{*}) M \xrightarrow{[1]}.
\]
Similar reasoning leads to the triangle $i_* (\beta_{\bleq 0}i^{!}) M \to M \to (\beta_{\bleq 0}j_{*}) j^*M \xrightarrow{[1]}$.
\end{proof}

\begin{prop}
The following two full subcategories of $\Dmix_\scS(X,\E)^\circ$ constitute a t-structure:
\begin{align*}
\Dmix_\scS(X,\E)^{\circ,\le 0} &=
\{ M \mid \text{$\beta_{\bgeq 0}i_s^*M \in \p\Dmix_\scS(X_s,\E)^{\le 0}$ for all $s \in \scS$} \}, \\
\Dmix_\scS(X,\E)^{\circ,\ge 0} &=
\{ M \mid \text{$\beta_{\bleq 0}i_s^!M \in \p\Dmix_\scS(X_s,\E)^{\ge 0}$ for all $s \in \scS$} \}.
\end{align*}
Moreover, if $\E = \K$ or $\F$, this t-structure is preserved by $\D$.
\end{prop}
\begin{proof}
Let us first treat the special case where $X$ consists of a single stratum $X_s$.  In this case, the definition reduces to $\Dmix_\scS(X,\E)^{\circ,\le 0} = \Dmix_\scS(X,\E)^\circ \cap \p\Dmix_\scS(X_s,\E)^{\le 0}$ and $\Dmix_\scS(X,\E)^{\circ,\ge 0} = \Dmix_\scS(X,\E)^\circ \cap \p\Dmix_\scS(X_s,\E)^{\ge 0}$.  Because $\beta_{\bleq 0}$ and $\beta_{\bgeq 0}$ are t-exact here (see Lemma~\ref{lem:beta-stratum}), these categories do indeed constitute a t-structure on $\Dmix_\scS(X,\E)^\circ$.

The proposition now follows by induction on the number of strata in $X$ using general properties of recollement; see~\cite[Th{\'e}or{\`e}me~1.4.10]{bbd}.
\end{proof}

We denote the heart of this t-structure by
\[
\Perv^\circ_\scS(X,\E) := \Dmix_\scS(X,\E)^{\circ,\le 0} \cap \Dmix_\scS(X,\E)^{\circ,\ge 0}.
\]
We saw in the course of the proof that on a single stratum, we have $\Perv^\circ_\scS(X_s,\E) = \Perv_\scS^\mix(X_s,\E) \cap \Dmix_\scS(X_s,\E)^\circ$, but this does not necessarily hold for larger varieties.

For another description of this t-structure, we introduce the objects
\[
\Delta^\circ_s := \beta_{\bgeq 0}j_!\uuE_{X_s}
\qquad\text{and}\qquad
\nabla^\circ_s := \beta_{\bleq 0}j_*\uuE_{X_s}.
\]
By adjunction, we have
\begin{equation}\label{eqn:pervo-alt}
\begin{aligned}
\Dmix_\scS(X,\E)^{\circ,\le 0} &=
\{ M \mid \text{for all $s \in \scS$ and $k < 0$, $\Hom(M, \nabla^\circ_s[k]) = 0$} \}, \\
\Dmix_\scS(X,\E)^{\circ,\ge 0} &=
\{ M \mid \text{for all $s \in \scS$ and $k > 0$, $\Hom(\Delta^\circ_s[k], M) = 0$} \}.
\end{aligned}
\end{equation}
Note that by definition we have $\Delta^\circ_s \in \Dmix_\scS(X,\E)^{\circ,\le 0}$ and $\nabla^\circ_s \in \Dmix_\scS(X,\E)^{\circ,\ge 0}$, but it is not clear in general whether $\Delta^\circ_s$ and $\nabla^\circ_s$ belong to $\Perv^\circ_\scS(X,\E)$.  

Let $\oo H^i: \Dmix_\scS(X,\E)^\circ \to \Perv^\circ_\scS(X,\E)$ denote the $i$-th cohomology functor with respect to this t-structure.  For $s \in \scS$, we put
\[
\IC^\circ_s := \mathrm{im} \bigl( \oo H^0(\Delta^\circ_s) \to \oo H^0(\nabla^\circ_s) \bigr),
\]
where the map is induced by the natural map $\dmix_s \to \nmix_s$.  If $\E$ is a field, then $\Perv^\circ_\scS(X,\E)$ is a finite-length category, and its simple objects are precisely the objects $\IC^\circ_s$.  Moreover, in this case, these objects are preserved by $\D$.

\subsection{Quasihereditary structure}

The description in~\eqref{eqn:pervo-alt} matches the framework of~\cite[Proposition~2(c)]{bezru2}.  That statement (see also~\cite[Remark~1]{bezru2}) tells us that when $\E$ is a field, $\Perv^\circ_\scS(X,\E)$ always satisfies ungraded analogues of axioms \eqref{it:qh-def-fin}--\eqref{it:qh-def-ker} of Definition~\ref{defn:qhered} (with respect to the objects $\Delta^\circ_s$ and  $\nabla^\circ_s$).
Under additional assumptions, we can obtain finer information about this category.

\begin{lem}\label{lem:perv-circ-qher}
Assume that $\E = \F$ or $\K$, and that for all $s \in \scS$, $\Delta^\circ_s$ and  $\nabla^\circ_s$ lie in $\Perv^\mix_\scS(X,\E)$.
\begin{enumerate} 
\item The category $\Perv^\circ_\scS(X,\E)$ is quasihereditary, the $\Delta^\circ_s$ and the $\nabla^\circ_s$ being, respectively, the standard and costandard objects. Moreover, if $\sT^\circ_\scS(X,\E)$ denotes the category of tilting objects in $\Perv^\circ_\scS(X,\E)$, the natural functors
\[
\Kb\sT^\circ_\scS(X,\E) \to \Db\Perv^\circ_\scS(X,\E) \to \Dmix_\scS(X,\E)^\circ
\]
are equivalences of categories. \label{it:pervo-std-costd}
\item We have $\Perv^\circ_\scS(X,\E) \subset \Perv^\mix_\scS(X,\E)$, and the inclusion functor $\Dmix_\scS(X,\E)^\circ \to \Dmix_\scS(X,\E)$ is t-exact.\label{it:pervo-incl}
\item If the objects $\cE^\mix_s$ are perverse, then they lie in $\Perv^\circ_\scS(X,\E)$, and they are precisely the indecomposable tilting objects therein.\label{it:pervo-tilt}
\end{enumerate}
\end{lem}
\begin{proof}
If all the objects $\Delta^\circ_s$ and $\nabla^\circ_s$ lie in $\Perv^\mix_\scS(X,\E)$, then there are no nonvanishing negative-degree $\Ext$-groups among them, so we see from~\eqref{eqn:pervo-alt} that these objects lie in $\Perv^\circ_\scS(X,\E)$.  Next, the proof of \cite[Lemma~3.2]{modrap2} is easily adapted to show that for any $s, t \in \scS$, we have
\[
\Hom_{\Dmix_\scS(X,\E)^\circ}(\Delta^\circ_s, \nabla^\circ_t[i]) = 0
\qquad\text{if $i \ne 0$.}
\]
With these observations in hand, the rest of the proof of part~\eqref{it:pervo-std-costd} is essentially identical to that of \cite[Proposition~3.10 and Lemma~3.14]{modrap2}.

We prove part~\eqref{it:pervo-incl} by induction on the number of strata in $X$.  If $X$ consists of a single stratum, the statement holds trivially.

Otherwise, choose an open stratum $X_s \subset X$.  It suffices to prove that every simple object of $\Perv^\circ_\scS(X,\E)$ lies in $\Perv^\mix_\scS(X,\E)$.  For $t \ne s$, the object $\IC^\circ_t$ is supported on the smaller variety $X \smallsetminus X_s$, so we know by induction that it lies in $\Perv^\mix_\scS(X,\E)$.  It remains to consider $\IC^\circ_s$.  Let $K$ be the kernel of the natural map $\Delta^\circ_s \to \IC^\circ_s$.  Since $K$ is also supported on $X \smallsetminus X_s$, we know that $K \in \Perv^\mix_\scS(X,\E)$.  By assumption, $\Delta^\circ_s \in \Perv^\mix_\scS(X,\E)$, so by considering the distinguished triangle $K \to \Delta^\circ_s \to \IC^\circ_s \xrightarrow{[1]}$, we see that $\IC^\circ_s \in \p\Dmix_\scS(X,\E)^{\le 0}$.  Since $\D(\IC^\circ_s) \cong \IC_s^\circ$, this object also lies in $\p\Dmix_\scS(X,\E)^{\ge 0}$, and hence in $\Perv^\mix_\scS(X,\E)$, as desired.

Finally, we consider part~\eqref{it:pervo-tilt}.  We claim that $\Hom(\cE^\mix_s, \nmix_t\{ n\}[k]) = 0$ for all $n < 0$.  When $k = 0$, this follows from the assumption that $\cE^\mix_s$ is perverse, and when $k \ne 0$, it follows from the same arguments as for Lemma~\ref{lem:weights-dmix-nmix}.
Thus, $\cE^\mix_s$ lies in $\Dmix_\scS(X,\E)_{\bleq 0}$.  Since $\D(\cE^\mix_s) \cong \cE^\mix_s$, this object also lies in $\Dmix_\scS(X,\E)_{\bgeq 0}$, hence in $\Dmix_\scS(X,\E)^\circ$.  Similar arguments show that
\[
\Hom(\cE^\mix_s, \nmix_t[k]) \cong \Hom(\cE^\mix_s,\nabla^\circ_t[k])
\]
vanishes for $k > 0$.  That condition and its dual together imply that $\cE^\mix_s$ belongs to $\Perv^\circ_\scS(X,\E)$ and is a tilting object therein, by, say, the criterion in~\cite[Lemma~4]{bezru2}.  The $\cE^\mix_s$ are indecomposable and parametrized by $\scS$, so they must coincide with the indecomposable tilting objects of $\Perv^\circ_\scS(X,\E)$.
\end{proof}

\subsection{A first positivity criterion}

We conclude this section with a result collecting a number of conditions equivalent to $\Perv^\mix_\scS(X,\E)$ being positively graded.  The proof makes use of Verdier duality, but no other tools coming from geometry.  Indeed, if $\cA$ is any graded quasihereditary category equipped with an antiautoequivalence satisfying similar formal properties to $\D$, one can formulate an analogue of the following proposition for $\Db(\cA)$.  The argument below will go through essentially verbatim.

\begin{prop}\label{prop:perv-pos}
Assume that $\E = \F$ or $\K$.  The following are equivalent:
\begin{enumerate}
\item The category $\Perv^\mix_\scS(X,\E)$ is positively graded.\label{it:perv-pos}
\item We have $[\dmix_s : \IC^\mix_t\la n\ra] = 0$ if $n > 0$.\label{it:delta-compos}
\item We have $(\cP^\mix_s : \dmix_t\la n\ra) = 0$ if $n > 0$.\label{it:proj-delta-filt}
\item We have $\IC^\mix_s \in \Dmix_\scS(X,\E)^\circ$ for all $s \in \scS$.\label{it:ic-mix-pure}
\item For all $n \in \Z$, the functors $\beta_{\bleq n}$ and $\beta_{\bgeq n}$ are t-exact for the perverse t-structure on $\Dmix_\scS(X,\E)$.\label{it:beta-exact}
\item We have $\IC^\circ_s \cong \IC^\mix_s$ for all $s \in \scS$.\label{it:ic-circ-mix}
\end{enumerate}
Moreover, if these conditions hold, then $\Perv^\circ_\scS(X,\E)$ can be identified with the Serre subcategory of $\Perv^\mix_\scS(X,\E)$ generated by all the $\IC^\mix_s$ (without Tate twists).
\end{prop}

\begin{rmk}
\label{rmk:dcirc-perverse}
The last assertion says that when the above conditions hold, we are in the setting of Proposition~\ref{prop:parshall-scott}; in this case the two definitions of $\Delta^\circ_s$ and of $\nabla^\circ_s$ coincide. Moreover, under this assumption all the objects $\Delta^\circ_s$ and $\nabla^\circ_s$ lie in $\Perv^\mix_\scS(X,\E)$, so the conclusions of Lemma~\ref{lem:perv-circ-qher} hold as well.
\end{rmk}

\begin{proof}
\eqref{it:perv-pos}${}\Longleftrightarrow{}$\eqref{it:delta-compos}${}\Longleftrightarrow{}$\eqref{it:proj-delta-filt}.  We saw in Proposition~\ref{prop:pos-qher} that~\eqref{it:perv-pos} holds if and only if both~\eqref{it:delta-compos} and~\eqref{it:proj-delta-filt} hold.  But by Verdier duality,~\eqref{it:delta-compos} holds if and only if $[\nmix_s\la n\ra : \IC^\mix_t] = 0 $ for all $n > 0$.  By the reciprocity formula, the latter is equivalent to~\eqref{it:proj-delta-filt}.

\eqref{it:perv-pos}${}\Longrightarrow{}$\eqref{it:ic-mix-pure}.  
As observed in the proof of~\eqref{eqn:pos-higher-ext}, $\IC^\mix_s$ admits a finite resolution
$\cdots \to P^{-1} \to P^0$ such that every term $P^i$ is a direct sum of objects of the form $\cP^\mix_t\la n\ra$ with $n \le 0$. 
Using~\eqref{it:proj-delta-filt}, we see that every term of this projective resolution lies in $\Dmix_\scS(X,\E)_{\bleq 0}$, so $\IC^\mix_s \in \Dmix_\scS(X,\E)_{\bleq 0}$ as well.  Since $\IC^\mix_s$ is stable under Verdier duality $\D$, we also have $\IC^\mix_s \in \Dmix_\scS(X,\E)_{\bgeq 0}$.

\eqref{it:ic-mix-pure}${}\Longrightarrow{}$\eqref{it:beta-exact}.  The assumption implies that
\begin{equation}\label{eqn:beta-bleq-formula}
\beta_{\bleq n}(\IC^\mix_s\la m\ra) \cong
\begin{cases}
\IC^\mix_s\la m\ra & \text{if $m \le n$,} \\
0 & \text{if $m > n$,}
\end{cases}
\end{equation}
along with a similar formula for $\beta_{\bgeq n}$.  Since $\beta_{\bleq n}$ and $\beta_{\bgeq n}$ send every simple object of $\Perv^\mix_\scS(X,\E)$ to an object of $\Perv^\mix_\scS(X,\E)$, they are both t-exact.

\eqref{it:beta-exact}${}\Longrightarrow{}$\eqref{it:ic-circ-mix}. First we note that, if~\eqref{it:beta-exact} holds, then the assumptions of Lemma~\ref{lem:perv-circ-qher} are satisfied. 
Consider the distinguished triangle
\[
\beta_{\bleq -1}\IC^\mix_s \to \IC^\mix_s \to \beta_{\bgeq 0}\IC^\mix_s \xrightarrow{[1]}.
\]
Since $\beta_{\bleq -1}$ and $\beta_{\bgeq 0}$ are exact, this is actually a short exact sequence in $\Perv^\mix_\scS(X,\E)$.  The middle term is simple, so either the first or last term must vanish.  The nonzero morphism $\IC^\mix_s \to \nmix_s$ shows that $\beta_{\bgeq 0}\IC^\mix_s \ne 0$.  Thus, we have $\beta_{\bleq -1}\IC^\mix_s = 0$, and $\IC^\mix_s \cong \beta_{\bgeq 0}\IC^\mix_s$.  A dual argument shows that we actually have
\[
\IC^\mix_s \cong \beta_{\bleq 0}\beta_{\bgeq 0} \IC^\mix_s.
\]
Moreover, applying the exact functor $\beta_{\bleq 0} \beta_{\bgeq 0}$ to the canonical morphism $\dmix_s \to \nmix_s$ tells us that $\IC^\mix_s$ is the image (in $\Perv^\mix_\scS(X,\E)$) of the map $\Delta^\circ_s \to \nabla^\circ_s$.  On the other hand, Lemma~\ref{lem:perv-circ-qher} tells us that this map is also a morphism in $\Perv^\circ_\scS(X,\E)$, where its image is $\IC^\circ_s$.  Since the inclusion functor $\Perv^\circ_\scS(X,\E) \to \Perv^\mix_\scS(X,\E)$ is exact (again by Lemma~\ref{lem:perv-circ-qher}), the image of $\Delta^\circ_s \to \nabla^\circ_s$ is the same in both categories.

\eqref{it:ic-circ-mix}${}\Longrightarrow{}$\eqref{it:perv-pos}.  The assumption implies that $\IC^\mix_s \in \Dmix_\scS(X,\E)_{\bleq 0}$, and that if $n > 0$, then $\IC^\mix_t\la n\ra[1] \in \Dmix_\scS(X,\E)_{\bgeq 1}$.  Therefore, 
\[
\Ext^1(\IC^\mix_s, \IC^\mix_t\la n\ra) = \Hom_{\Dmix_\scS(X,\E)}(\IC^\mix_s, \IC^\mix_t\la n\ra[1]) = 0 
\]
by Lemma~\ref{lem:baric-basic}\eqref{it:baric-orth}.
By Proposition~\ref{prop:pos-qher}, it follows that $\Perv^\mix_\scS(X,\E)$ is positively graded.

The last assertion in the proposition is immediate from part~\eqref{it:ic-circ-mix}.
\end{proof}

\subsection{Koszulity}

For later use,
we conclude this section with a description of the most favorable situation. (See~\cite[Proposition~5.7.2]{rsw} and~\cite[Theorem~5.3]{weidner} for related results.)

\begin{cor}
\label{cor:parity-koszul}
Assume that $\E=\K$ or $\F$, and that for all $s \in \scS$ we have $\IC_s^\mix \cong \cE^\mix_s$. Then the category $\Perv^\mix_\scS(X,\E)$ is Koszul (and hence in particular positively graded).
\end{cor}

\begin{proof}
Under our assumptions we have
\[
\Ext^k_{\Perv^\mix_\scS(X,\E)}(\IC_s^\mix, \IC^\mix_t \la n \ra) \cong \Hom_{\Dmix_{\scS}(X,\E)}(\cE^\mix_s, \cE^\mix_t \{-n\}[k+n]),
\]
which clearly vanishes unless $k+n=0$.
\end{proof}

\begin{rmk}
One can easily show that, under these assumptions, $\Perv^\mix_\scS(X,\E)$ is even standard Koszul.
\end{rmk}

\section{Further study of mixed perverse $\O$-sheaves}
\label{sec:ico}

We continue in the setting of Section~\ref{sec:weights}, with the goal of furthering our understanding of positivity.  The arguments in the previous section were mostly based on general principles of homological algebra, and in some cases were restricted to field coefficients.  To make further progress, we need to bring in concrete geometric facts about our variety.  In this section, we will focus on $\O$-sheaves as an intermediary between $\F$- and $\K$-sheaves, and
the main results will involve the assumption that $\IC^\mix_s(\K) \cong \cE^\mix_s(\K)$.  This holds, of course, on flag varieties, by~\cite{kl}.

\subsection{Describing extensions from an open set}
\label{ss:extension}

We begin with a brief review of a convenient language for describing objects in $\Dmix_\scS(X,\E)$ with a specified restriction to some open subset of $X$ (see e.g.~\cite[Lemma~2.18]{jmw} for a similar statement in the classical setting).  The descriptions below are valid for arbitrary coefficients, although they will be used in this paper mainly in the case where $\E = \O$.

Let $X_t \subset X$ be a closed stratum, and let $j: U \hookrightarrow X$ be the inclusion of the complementary open subset.  Let $M_U \in \Dmix_\scS(U,\E)$.  Then there is a bijection between the set of isomorphism classes of pairs $(M,\alpha)$ where $M \in \Dmix_\scS(X,\E)$ and $\alpha : j^* M \simto M_U$ is an isomorphism in $\Dmix_\scS(U,\E)$,
and the set of isomorphism classes of distinguished triangles
\begin{equation}
\label{eqn:triangle-extensions}
A \to i_t^{!}j_{!}M_U \to B \xrightarrow{[1]}
\end{equation}
in $\Dmix_{\scS}(X_t, \E)$.
Specifically, given such a triangle, one can recover $M$ as the cone of the composite morphism $i_{t*}A \to i_{t*}i_t^{!}j_{!}M_U \to j_{!}M_U$. On the other hand, to $M$ we associate the natural triangle with
\[
A = i_{t}^*M[-1]
\qquad\text{and}\qquad
B = i_{t}^!M.
\]
Here are some specific examples:

\subsubsection{}\label{sss:ic}
If $M_U$ is perverse, the extension $M = j_{!*}M_U$ corresponds to
\[
A = \tau_{\le 0} i_t^{!}j_{!}M_U,
\qquad
B = \tau_{\ge 1} i_t^{!}j_{!}M_U
\]
(see~\cite[Proposition~1.4.23]{bbd}).

\subsubsection{}\label{sss:!}
The extension $M = j_{!}M_U$ corresponds to $A = 0$, $B = i_t^{!}j_{!}M_U$.  The extension $M = j_{*}M_U$ corresponds to $A = i_t^{!}j_{!}M_U$, $B = 0$.

\subsubsection{}\label{sss:beta}
If $M_U \in \Dmix_\scS(U,\E)^\circ$, then $\beta_{\bgeq 0}j_{!}M_U$ corresponds to
\[
A = \beta_{\bleq -1}i_t^{!}j_{!}M_U,
\qquad
B = \beta_{\bgeq 0}i_t^{!}j_{!}M_U.
\]
(Indeed we have $A=i_t^* \beta_{\bgeq 0}j_{!}M_U [-1]$, hence $\beta_{\bgeq 0} A = (\beta_{\bgeq 0} i_t^*)(\beta_{\bgeq 0} j_!) M_U[-1]=0$, which implies that $A$ is in $\Dmix_{\scS}(X_t, \E)_{\bleq -1}$. On the other hand, $B=i_t^! \beta_{\bgeq 0}j_{!}M_U$ is in $\Dmix_{\scS}(X_t, \E)_{\bgeq 0}$ by Lemma~\ref{lem:baric-stability}. Hence triangle~\eqref{eqn:triangle-extensions} must be the truncation triangle for the baric structure.)

\subsubsection{}\label{sss:parity}
If $M_U \in \Parity_{\scS}(U,\E)$, then $i_t^{!}j_{!}M_U \in \Dmix_{\scS}(X_t,\E)$ has weights in the interval $[-1,0]$.  In other words, it can be written as a complex $F^\bullet$ in $\Kb\Parity_{\scS}(X_t,\E)$ in which the only nonzero terms are $F^0$ and $F^1$. If $\E=\K$ or $\F$, then the ``parity extension'' of $M_U$ constructed in~\cite[Lemma~2.27]{jmw} (considered as an object in $\Dmix_\scS(X,\E)$) corresponds to
\[
A = F^1[-1], \qquad B = F^0.
\]

\subsection{Stalks of the $\Delta^\circ_s(\O)$}
\label{ss:stalks-deltao}

If $M$ is in $\Dmix_{\scS}(X,\E)$, we will say that \emph{the stalks of $M$ are pure of weight $0$} if for all $s \in \scS$, the object $i_s^* M \in \Dmix_{\scS}(X_s,\E)$ is pure of weight $0$, i.e.~a direct sum of objects of the form $\uuE_{X_s}\{i\}$ for $i \in \Z$. Typical objects that satisfy this condition are the parity sheaves $\cE^\mix_s$. Note that if $M$ is in $\Dmix_{\scS}(X,\O)$, then the stalks of $M$ are pure of weight $0$ iff the stalks of $\F(M)$ are pure of weight~$0$.

In the proofs below we will use the following notation. Recall from Lemma~\ref{lem:beta-stratum} that on a single stratum $X_s$, the functors $\beta_{\bleq n}$ and $\beta_{\bgeq n}$ are t-exact.  For objects in $\Dmix_\scS(X_s,\E)$, we set
\[
\pH^k_r := \pH^k \circ \beta_{\bleq r} \circ \beta_{\bgeq r} \cong \beta_{\bleq r} \circ \beta_{\bgeq r} \circ \pH^k.
\]

The following result relates ``pointwise purity'' to a ``torsion-free'' condition.

\begin{lem}\label{lem:stalks-dcirc-ic}
For each $s \in \scS$, the following conditions are equivalent:
\begin{enumerate}
\item The stalks of $\Delta^\circ_s(\F)$ are pure of weight $0$.\label{it:stalk-f}
\item The stalks of $\Delta^\circ_s(\O)$ are pure of weight $0$.\label{it:stalk-o}
\end{enumerate}
Moreover, if $\IC^\mix_s(\K) \cong \cE^\mix_s(\K)$, then these statements are also equivalent to the following one:
\begin{enumerate}
\setcounter{enumi}{2}
\item We have $\Delta^\circ_s(\O) \cong \IC^\mix_s(\O)$, and the stalks of $\IC^\mix_s(\O)$ are free.\label{it:stalk-ic}
\end{enumerate}
\end{lem}
\begin{proof}
Conditions~\eqref{it:stalk-f} and~\eqref{it:stalk-o} are equivalent because $\F(\Delta^\circ_s(\O))\cong\Delta^\circ_s(\F)$ (see Lemma~\ref{lem:beta-K-F}).   

Assume now that $\IC^\mix_s(\K) \cong \cE^\mix_s(\K)$. If condition~\eqref{it:stalk-ic} holds, then the stalks of $\IC^\mix_s(\O)$ are pure of weight $0$, since those of $\K(\IC^\mix_s(\O)) \cong \IC^\mix_s(\K)$ are, implying~\eqref{it:stalk-o}.

Conversely, suppose that condition~\eqref{it:stalk-o} holds.  We will prove condition~\eqref{it:stalk-ic} by induction on the number of strata in $X$.  If $X$ consists of a single stratum, the statement is trivial.  Otherwise, let $X_t \subset X$ be a closed stratum, and let $j: U \hookrightarrow X$ be the inclusion of the complementary open subset.  Let $X_s$ be a stratum in $U$.  Let $M_U := \Delta^\circ_{U,s}(\O) \cong \IC^\mix_{U,s}(\O)$, and let $L = i_t^{!} j_{!}M_U$.

We begin by showing that $\pH^0_0(L)$ is a torsion $\O$-module.  Observe that $\K(L) \cong i_t^{!}j_{!} (\K(\IC^\mix_{U,s}(\O))) \cong i_t^{!} j_{!} \IC^\mix_{U,s}(\K)$.  According to \S\ref{sss:ic}, we have $\tau_{\le 0}(\K(L)) \cong i_t^{*}\IC^\mix_s(\K)[-1] \cong i_t^{*}\cE^\mix_s(\K)[-1]$. 
It follows that, when $k \leq 0$, $\pH^k_r(\K(L))$ vanishes unless $r = k-1$.  In particular, $\pH^0_0(\K(L)) \cong \K(\pH_0^0(L)) = 0$.  This implies that $\pH^0_0(L)$ is torsion.

Next, we carry out a similar line of reasoning using the fact that $\beta_{\bleq -1}L \cong i_t^{*}\Delta^\circ_{X,s}(\O)[-1]$ (see~\S\ref{sss:beta}).  The latter is pure of weight $-1$ by assumption so, if $r \le  -1$, $\pH^k_r(L)$ vanishes unless $k = r+1$.  In particular, $\pH^k_r(L)$ vanishes for all $k > 0$ when $r \le -1$.  In other words, $\beta_{\bleq -1}L \in \p\Dmix_\scS(X_t,\O)^{\le 0}$.

Finally, assumption~\eqref{it:stalk-o} implies that $\Delta^\circ_{U,s}(\O)$ has weights${}\le 0$ (see Lemma~\ref{lem:wt-stalks}), and so $L$ has weights${}\le 0$ as well (see Lemma~\ref{lem:wt-stability}).  That is, $\pH^k_r(L) = 0$ for $k < r$, and it must be free when $k = r$.  But we previously saw that $\pH^0_0(L)$ is torsion, so in fact, it must vanish.   For $r \ge 1$, we have that $\pH^k_r(L) = 0$ for all $k \le 0$.  Combining these, we find that $\beta_{\bgeq 0}L \in \p\Dmix_\scS(X_t,\O)^{\ge 1}$.  This fact, together with the previous paragraph, tells us that the two distinguished triangles
\[
\beta_{\bleq -1}L \to L \to \beta_{\bgeq 0}L \to \quad \text{and} \quad
\tau_{\le 0}L \to L \to \tau_{\ge 1}L \to
\]
coincide.  From the discussion in~\S\ref{sss:ic} and~\S\ref{sss:beta}, we conclude that $\Delta^\circ_{X,s}(\O) \cong \IC^\mix_{X,s}(\O)$.  The stalks of $\IC^\mix_{X,s}(\O)$ are torsion-free because those of $\Delta^\circ_{X,s}(\O)$ are by assumption.
\end{proof}

\subsection{Another positivity criterion}
\label{ss:mixed-perv}

The main result of this section is the following.

\begin{thm}\label{thm:positivity}
Assume that $\IC^\mix_s(\K) \cong \cE^\mix_s(\K)$ for all $s \in \scS$.  Then the following are equivalent:
\begin{enumerate}
\item $\Perv^\mix_\scS(X,\F)$ is positively graded.\label{it:pos-f}
\item For all $s, t \in \scS$, we have $[\F(\IC^\mix_s(\O)) : \IC^\mix_t(\F)\la n\ra] = 0$ unless $n = 0$.\label{it:fic-compos}
\item For all $s \in \scS$, $\K(\cP^\mix_s(\O))$ is a direct sum of objects of the form $\cP^\mix_t(\K)$ (i.e., without Tate twists).\label{it:kps-compos}
\item For all $s \in \scS$, we have $\Delta^\circ_s(\O) \cong \IC^\mix_s(\O)$.\label{it:delta-ico}
\end{enumerate}
\end{thm}
\begin{proof}[Proof of the equivalence of parts~\eqref{it:pos-f}--\eqref{it:kps-compos}]
We begin by proving the equivalence of parts~\eqref{it:fic-compos} and~\eqref{it:kps-compos}. By the same arguments as in the proof of~\cite[Lemma~5.2]{modrap1} (see also~\cite[Lemma~2.10]{modrap2}), the $\O$-module $\Hom(\cP^\mix_s(\O),\IC^\mix_t(\O) \la n \ra)$ is free, and we have natural isomorphisms
\begin{align*}
\F \otimes_\O \Hom(\cP^\mix_s(\O),\IC^\mix_t(\O) \la n \ra) & \cong \Hom (\cP^\mix_s(\F), \F(\IC^\mix_t(\O)) \la n \ra), \\
\K \otimes_\O \Hom(\cP^\mix_s(\O),\IC^\mix_t(\O) \la n \ra) & \cong \Hom (\K(\cP^\mix_s(\O)), \IC^\mix_t(\K) \la n \ra).
\end{align*}
Condition~\eqref{it:fic-compos} expresses the property that the first vector space can be nonzero only if $n=0$, and condition~\eqref{it:kps-compos} expresses the property that the second vector space can be nonzero only if $n=0$. Hence these conditions are indeed equivalent.

To prove the other equivalences we need to introduce Grothendieck groups. For $\E=\K$, $\O$ or $\F$, consider the Grothendieck group $K^\mix_\scS(X,\E)$ of the abelian category $\Perv^\mix_\scS(X,\E)$. This abelian group naturally has the structure of a $\Z[v,v^{-1}]$-module, where $v$ acts via the shift $\la 1 \ra$. The classes of the objects $\IC^\mix_s(\E)$ form a basis of this $\Z[v,v^{-1}]$-module, and similarly for the objects $\dmix_s(\E)$. (When $\E = \O$, this assertion relies on the fact that $\E$ has finite global dimension.) Moreover, the functors $\K(-)$ and $\F(-)$ induce morphisms of $\Z[v,v^{-1}]$-modules
\[
e_\K : K^\mix_\scS(X,\O) \to K^\mix_\scS(X,\K), \qquad r_\F : K^\mix_\scS(X,\O) \to K^\mix_\scS(X,\F).
\]
For any $s \in \scS$, write
\[
[\dmix_s(\O)] = \sum_{t \in \scS} d_{s,t} [\IC^\mix_t(\O)]
\]
where $d_{s,t} \in \Z[v,v^{-1}]$. 

Now we can prove that~\eqref{it:fic-compos} implies~\eqref{it:pos-f}.
First, it follows from our assumption that $\Perv^\mix_\scS(X,\K)$ is positively graded (see Corollary~\ref{cor:parity-koszul}).  Therefore, applying $e_\K$, we see that we must have $d_{s,t} \in \Z[v^{-1}]$ for any $s,t$. Now assumption \eqref{it:fic-compos} ensures that
\[
r_\F(\IC^\mix_t(\O)) \in \sum_{u \in \scS} \Z \cdot [\IC^\mix_u(\F)].
\]
It follows that $[\dmix_s(\F)] = r_\F([\dmix_s(\O)])$ is a $\Z[v^{-1}]$-linear combination of the $[\IC^\mix_u(\F)]$.  In other words, Proposition~\ref{prop:perv-pos}\eqref{it:delta-compos} holds, so $\Perv^\mix_\scS(X,\F)$ is positively graded.

For the converse, suppose that~\eqref{it:pos-f} holds. Write
\[
[\cP^\mix_s(\O)] = \sum_{t \in \scS} p_{s,t} [\dmix_t(\O)]
\]
where $p_{s,t} \in \Z[v,v^{-1}]$. Applying $r_\F$, we obtain that $p_{s,t} \in \Z[v^{-1}]$. Since $\Perv^\mix_\scS(X,\K)$ is also positively graded, we deduce that the indecomposable direct summands of $\K(\cP^\mix_s(\O))$ are of the form $\cP^\mix_t(\K) \la n \ra$ with $n \leq 0$. Assume that $\cP^\mix_t(\K) \la n \ra$ appears for some $n<0$. By the remarks in the equivalence of \eqref{it:fic-compos} and \eqref{it:kps-compos}, this implies that $\IC_s^\mix \la -n \ra$ is a composition factor of the mixed perverse sheaf $\F(\IC^\mix_t(\O))$. Then $\IC_s^\mix \la -n \ra$ is also a composition factor of $\F(\dmix_t(\O))=\dmix_t(\F)$, which contradicts Proposition~\ref{prop:perv-pos}\eqref{it:delta-compos}.
\end{proof}

\begin{rmk}
Since $\IC^\mix_s(\O)$ and $\D(\IC^\mix_s(\O))$ differ only by torsion, the mixed perverse sheaves $\F(\IC_s^\mix(\O) )$ and $\D \bigl( \F(\IC_s^\mix(\O) \bigr)$ have the same composition factors. Hence condition \eqref{it:fic-compos} is equivalent to the property that all composition factors of all $\F(\IC^\mix_s(\O))$ are of the form $\IC^\mix_t(\F)\la n \ra$ with $n\leq 0$, or all of the form $\IC^\mix_t(\F)\la n \ra$ with $n\geq 0$. A similar remark applies to~\eqref{it:kps-compos}.
\end{rmk}

\begin{lem}\label{lem:pervo-pos}
Assume that $\IC^\mix_s(\K) \cong \cE^\mix_s(\K)$ for all $s \in \scS$.  In addition, assume that  conditions~\eqref{it:pos-f}--\eqref{it:delta-ico} of Theorem~{\rm \ref{thm:positivity}} hold for every locally closed union of strata $Y \subsetneq X$.  Then, for all $s \in \scS$, the objects $\beta_{\bgeq 0}\IC^\mix_s(\F)$, $\beta_{\bleq 0}\IC^\mix_s(\F)$, $\Delta^\circ_s(\F)$, and $\nabla^\circ_s(\F)$ are all perverse.
\end{lem}
\begin{proof}
If $X_s$ is not an open stratum, then the objects in question are all supported on a proper closed subvariety of $X$, and so are perverse by assumption and Proposition~\ref{prop:perv-pos}.  Assume henceforth that $X_s$ is an open stratum, and let $Y = X \smallsetminus X_s$.  We will treat $\beta_{\bgeq 0}\IC^\mix_s(\F)$ and $\Delta^\circ_s(\F)$; the statement follows for the other two objects by Verdier duality.

Let $Q$ denote the cokernel of the map $\IC^\mix_s(\F) \to \nmix_s(\F)$. Since $Q$ is supported on $Y$, Proposition~\ref{prop:perv-pos}\eqref{it:beta-exact} tells us that the triangle
\[
\beta_{\bleq -1}Q \to Q \overset{h}{\to} \beta_{\bgeq 0}Q \xrightarrow{[1]}
\]
is actually a short exact sequence in $\Perv^\mix_\scS(X,\F)$.  In particular, the map $h$ is surjective.  Now consider the commutative diagram
\[
\xymatrix{
\IC^\mix_s(\F) \ar[r]\ar[d] & \nmix_s(\F) \ar[r]^-{p} \ar@{=}[d] & Q \ar[d]^-{h} \\
\beta_{\bgeq 0}\IC^\mix_s(\F) \ar[r] & \beta_{\bgeq 0}\nmix_s(\F) \ar[r]_-{q} & \beta_{\bgeq 0}Q }
\]
Since $h$ and $p$ are both surjective maps in $\Perv^\mix_\scS(X,\F)$, $q$ is as well.  It follows that the cocone of $q$ (i.e.~$\beta_{\bgeq 0}\IC^\mix_s(\F)$) lies in $\Perv^\mix_\scS(X,\F)$.

Next, let $K$ denote the kernel of the map $\dmix_s(\F) \to \IC^\mix_s(\F)$, and form the distinguished triangle
\[
\beta_{\bgeq 0}K \to \beta_{\bgeq 0}\dmix_s(\F) \to \beta_{\bgeq 0}\IC^\mix_s(\F) \xrightarrow{[1]}.
\]
Since $K$ is supported on $Y$, Proposition~\ref{prop:perv-pos}\eqref{it:beta-exact} again tells us that the first term lies in $\Perv^\mix_\scS(X,\F)$.  We have just seen above that the last term also lies in $\Perv^\mix_\scS(X,\F)$, so the middle term (which is $\Delta^\circ_s(\F)$ by definition) does as well.
\end{proof}

\begin{proof}[End of the proof of Theorem~{\rm \ref{thm:positivity}}]
We will show that condition~\eqref{it:delta-ico} is equivalent to condition~\eqref{it:ic-circ-mix} of Proposition~\ref{prop:perv-pos}, by induction on the number of strata in $X$.  If $X$ consists of a single stratum, it is clear that both statements are true.

Otherwise, let $X_s \subset X$ be an open stratum, and let $X_t \subset X$ be a closed stratum.  Let $U = X \smallsetminus X_t$ and $Y = X \smallsetminus X_s$.  Note that if either~\eqref{it:delta-ico} or condition~\eqref{it:ic-circ-mix} of Proposition~\ref{prop:perv-pos} holds on $X$, the same statement holds on both $U$ and $Y$, and hence, by induction, all parts of Theorem~\ref{thm:positivity} hold on both $U$ and $Y$.  For the remainder of the proof, we assume that this is the case.  We must show that $\Delta^\circ_s(\O) \cong \IC^\mix_s(\O)$ if and only if $\IC^\circ_s(\F) \cong \IC^\mix_s(\F)$.  By Lemma~\ref{lem:pervo-pos}, $\beta_{\bgeq 0}\IC^\mix_s(\F)$ and $\Delta^\circ_s(\F)$ are perverse.  

For $\E=\K$, $\O$ or $\F$, let $M_U(\E) := \Delta^\circ_{U,s}(\E)$. Note that $\F(M_U(\O)) \cong M_U(\F)$ and $\K(M_U(\O)) \cong M_U(\K)$ (see Lemma~\ref{lem:beta-K-F}), and that $M_U(\E) \cong \IC^\mix_{U,s}(\E)$ if $\E=\K$ or $\O$.  Let $j: U \hookrightarrow X$ be the inclusion map, and let $L(\E) = i_t^{!}j_{!}M_U(\E)$. Since $\F(L(\O))=L(\F)$, there is a natural short exact sequence of $\F$-vector spaces
\begin{equation}\label{eqn:lfo-ses}
0 \to \F \otimes_\O \pH^k(L(\O)) \to \pH^k(L(\F)) \to \mathrm{Tor}^\O_1(\F,\pH^{k+1}(L(\O))) \to 0.
\end{equation}
On the other hand, we have $M_U(\K) \cong \IC^\mix_{U,s}(\K) \cong \cE^\mix_{U,s}(\K)$.  By assumption, $j_{!*}M_U(\K)$ coincides with the parity extension $\cE^\mix_{s}(\K)$ of $M_U(\K)$.  Comparing the constructions in~\S\ref{sss:ic} and~\S\ref{sss:parity}, we see that $\tau_{\le 0}L(\K)[1]$ and $\tau_{\ge 1}L(\K)$ are parity sheaves.  In other words,
\begin{equation}\label{eqn:ick-stalk}
\pH^k_r(L(\K)) = 0 \qquad\text{unless}\qquad
\begin{cases}
\text{$k \le 0$ and $r = k - 1$, or} \\
\text{$k \ge 1$ and $r = k$.}
\end{cases}
\end{equation}
We now proceed in several steps.

\emph{Step 1. If $k > 1$, then $\pH^k(\beta_{\bleq -1}L(\O)) = 0$.  If $k < 1$, then $\pH^k(\beta_{\bgeq 0}L(\O)) = 0$.}  Recall that $\Delta^\circ_s(\O) \cong \beta_{\bgeq 0}j_!M_U$.  From~\S\ref{sss:beta}, we have
\[
\beta_{\bleq -1}L(\O) \cong i_t^*\Delta^\circ_s(\O)[-1]
\qquad\text{and}\qquad
\beta_{\bgeq 0}L(\O) \cong i_t^!\Delta^\circ_s(\O).
\]
Since $\Delta^\circ_s(\F) \cong \F(\Delta^\circ_s(\O))$ is perverse, we have by~\cite[Lemma~3.5]{modrap2} that $\Delta^\circ_s(\O)$ lies in $\Perv^\mix_\scS(X,\O)$.  This implies that $i_t^*\Delta^\circ_s(\O)[-1] \in \p\Dmix_\scS(X_t,\O)^{\le 1}$, or in other words, $\pH^k(\beta_{\bleq -1}L(\O)) = 0$ for $k > 1$.

We likewise have $i_t^!\Delta^\circ_s(\O) \in \p\Dmix_\scS(X_t,\O)^{\ge 0}$.  We claim, furthermore, that $\pH^0(i_t^!\Delta^\circ_s(\O))$ is torsion-free: otherwise, $\F(i_t^!\Delta^\circ_s(\O)) \cong i_t^!\Delta^\circ_s(\F)$ would fail to lie in $\p\Dmix_\scS(X_t,\F)^{\ge 0}$, contradicting the fact that $\Delta^\circ_s(\F)$ is perverse.  To reiterate, $\pH^k(\beta_{\bgeq 0}L(\O))$ vanishes for $k < 0$ and is torsion-free when $k = 0$.  But it follows from~\eqref{eqn:ick-stalk} that $\pH^0(\beta_{\bgeq 0}L(\K)) \cong \K \otimes_\O \pH^0(\beta_{\bgeq 0}L(\O))$ vanishes.  Therefore, $\pH^0(\beta_{\bgeq 0}L(\O)) = 0$ as well, finishing the proof of Step 1.

\emph{Step 2. We have $\pH^0(\beta_{\bleq -1} i_t^{*}\Delta^\circ_s(\F)) \cong \F \otimes_\O \pH^1(\beta_{\bleq -1}L(\O))$.}  From Step~1, we know that $\pH^2(\beta_{\bleq -1}L(\O)) = 0$, so~\eqref{eqn:lfo-ses} tells us that $\pH^1(\beta_{\bleq -1}L(\F)) \cong \F \otimes_\O \pH^1(\beta_{\bleq -1}L(\O))$.  On the other hand, as in Step 1, $\beta_{\bleq -1}L(\F) \cong i_s^{*}\Delta^\circ_s(\F)[-1]$, and the result follows.

\emph{Step 3. We have $\Delta^\circ_s(\O) \cong \IC^\mix_s(\O)$ if and only if $\pH^1(\beta_{\bleq -1}L(\O)) = 0$.}  From the descriptions in~\S\ref{sss:ic} and~\S\ref{sss:beta}, we see that $\Delta^\circ_s(\O) \cong \IC^\mix_s(\O)$ if and only if
\[
\tau_{\le 0}L(\O) \cong \beta_{\bleq -1}L(\O)
\qquad\text{and}\qquad
\tau_{\ge 1}L(\O) \cong \beta_{\bgeq 0}L(\O).
\]
According to Step~1, we always have $\beta_{\bleq -1}L(\O) \in {}^p\Dmix_\scS(X_t,\O)^{\le 1}$ and $\beta_{\bgeq 0}L(\O) \in {}^p\Dmix_\scS(X_t,\O)^{\ge 1}$.  Thus, the conditions above hold if and only if $\pH^1(\beta_{\bleq -1}L(\O)) = 0$.

\emph{Step 4. We have $\IC^\circ_s(\F) \cong \IC^\mix_s(\F)$ if and only if $\pH^k(\beta_{\bleq -1}i_t^{*}\IC^\circ_s(\F)) = 0$ for all $k \ge 0$.}  We already know that the restrictions of $\IC^\circ_s(\F)$ and $\IC^\mix_s(\F)$ to $U$ agree.  Recall that $\IC^\mix_s(\F)$ is characterized (among all objects whose restriction to $U$ is $\IC^\mix_{U,s}(\F)$) by the following two properties:
\begin{equation}\label{eqn:ic-mix-char}
i_t^{*}\IC^\mix_s(\F) \in {}^p\Dmix_\scS(X_t,\F)^{\le -1}
\quad\text{and}\quad
i_t^{!}\IC^\mix_s(\F) \in {}^p\Dmix_\scS(X_t,\F)^{\ge 1}.
\end{equation}
Since $\IC^\circ_s(\F)$ is self-Verdier-dual, if it satisfies one of these properties then it must satisfy both.  Thus, $\IC^\circ_s(\F) \cong \IC^\mix_s(\F)$ if and only if
\[
\pH^k(i_t^{*}\IC^\circ_s(\F)) = 0 \qquad\text{for $k \ge 0$.}
\]
But $\IC^\circ_s(\F)$ is itself characterized by similar properties to those in~\eqref{eqn:ic-mix-char}, coming from the recollement structure in Proposition~\ref{prop:circ-recollement}.  In particular, we have $\pH^k(\beta_{\bgeq 0}i_t^{*}\IC^\circ_s(\F)) = 0$ for $k \ge 0$.  By Lemma~\ref{lem:beta-stratum}, we deduce that for $k \geq 0$ we have $\pH^k(i_t^{*}\IC^\circ_s(\F)) \cong \pH^k(\beta_{\bleq -1}i_t^{*}\IC^\circ_s(\F))$, which finishes the proof of Step 4.

\emph{Step 5. We have $\pH^k(\beta_{\bleq -1}i_t^{*}\Delta^\circ_s(\F)) \cong \pH^k(\beta_{\bleq -1}i_t^{*}\IC^\circ_s(\F))$ for $k \ge 0$.}  Let $K$ be the kernel of the map $\Delta^\circ_s(\F) \to \IC^\circ_s(\F)$.  This kernel is to be taken in $\Perv^\circ_\scS(X,\F)$: we do not know at the moment whether $\IC^\circ_s(\F)$ lies in $\Perv^\mix_\scS(X,\F)$.  However, we do know that $K$ lies in $\Perv^\mix_\scS(X,\F)$, because $K$ is supported on $Y$, where the conclusions of Lemma~\ref{lem:perv-circ-qher} hold.  In fact, for each composition factor $\IC^\circ_u(\F) \cong \IC^\mix_u(\F)$ of $K$, we have
\[
\pH^k(\beta_{\bleq -1}i_t^{*}\IC^\circ_u(\F)) = 0 \qquad\text{for $k \ge 0$.}
\]
(If $u \ne t$, this holds because $i_t^{*}\IC^\circ_u(\F) \in \p\Dmix_\scS(X_t,\F)^{\le -1}$; if $u = t$, we clearly have $\beta_{\bleq -1}i_t^{*}\IC^\circ_u(\F) = 0$.)  Therefore, $\pH^k(\beta_{\bleq -1}K) = 0$ for $k \ge 0$.  The result follows from the long exact sequence in perverse cohomology associated with 
\[
\beta_{\bleq -1}i_t^{*}K \to \beta_{\bleq -1}i_t^{*}\Delta^\circ_s(\F) \to \beta_{\bleq -1}i_t^{*}\IC^\circ_s(\F) \xrightarrow{[1]}.
\]

\emph{Conclusion of the proof.}  Since $\Delta^\circ_s(\F)$ is perverse, we know that $\pH^k(i_t^{*}\Delta^\circ_s(\F)) = 0$ for $k > 0$, and so $\pH^k(\beta_{\bleq -1}i_t^{*}\Delta^\circ_s(\F)) = 0$ for $k > 0$ as well.  Then, Step~5 implies that $\pH^k(\beta_{\bleq -1}i_t^{*}\IC^\circ_s(\F)) = 0$ for $k > 0$, so we can rephrase Step~4 as follows: $\IC^\circ_s(\F) \cong \IC^\mix_s(\F)$ if and only if $\pH^0(\beta_{\bleq -1}i_t^{*}\IC^\circ_s(\F)) = 0$.  Using Step~5 again together with Step~2, we have that $\IC^\circ_s(\F) \cong \IC^\mix_s(\F)$ if and only if $\F \otimes_\O \pH^1(\beta_{\bleq -1}L(\O)) = 0$.  The latter holds if and only if $\pH^1(\beta_{\bleq -1}L(\O)) = 0$, and then Step~3 lets us conclude.
\end{proof}

\begin{cor}\label{cor:Q-Koszul}
Assume that $\IC^\mix_s(\K) \cong \cE^\mix_s(\K)$ for all $s \in \scS$.
Then the following conditions are equivalent:
\begin{enumerate}
\item The category $\Perv^\mix_{\scS}(X,\F)$ is standard $Q$-Koszul.\label{it:q-koszul-X}
\item For all $s \in \scS$, we have $\Delta^\circ_s(\O) \cong \IC^\mix_s(\O)$, and $\IC^\mix_s(\O)$ has torsion-free stalks.\label{it:torsion-free-X}
\end{enumerate}
\end{cor}

\begin{proof}
Each of these conditions independently implies that all parts of Theorem~\ref{thm:positivity} and of Proposition~\ref{prop:perv-pos} hold for $X$.  In particular, both conditions imply at least that $\Perv^\mix_{\scS}(X,\F)$ is positively graded, and that the perverse-sheaf meaning of the notation $\Delta^\circ_s$ is compatible its usage in Definition~\ref{defn:q-koszul}.  By Verdier duality, standard $Q$-Koszulity can be checked by a one-sided condition: $\Perv^\mix_{\scS}(X,\F)$ is standard $Q$-Koszul if and only if $\Ext^k(\Delta^\circ_s(\F), \nmix_t(\F)\la n\ra) = 0$ whenever $n \ne -k$.  By adjunction, the latter holds if and only if the stalks of $\Delta^\circ_s(\F)$ are pure of weight~$0$ for all $s$. That condition is equivalent to~\eqref{it:torsion-free-X} by Lemma~\ref{lem:stalks-dcirc-ic}, as desired.
\end{proof}

\section{Positivity and $Q$-Koszulity for flag varieties}
\label{sec:koszulity}

\subsection{Definitions and notation}

In this section we choose a connected reductive algebraic group $G$, a Borel subgroup $B \subset G$ and a maximal torus $T \subset B$, and focus on the case where $X=\cB:=G/B$ is the flag variety of $G$, endowed with the stratification by Bruhat cells (i.e.~by orbits of $B$).  We use the symbol ``$(B)$'' to denote this stratification. The strata are parametrized by the Weyl group $W:=N_G(T)/T$ of $G$; the dimension of $\cB_w$ is the length $\ell(w)$ of $w$ (for the natural Coxeter group structure on $W$ determined by our choice of $B$). By~\cite[\S4]{modrap2}, the assumptions at the beginning of Section~\ref{sec:weights} are satisfied in this setting. As in~\cite{modrap2} we will assume that $\ell$ is good for $G$. Note also that the assumption of Lemma~\ref{lem:stalks-dcirc-ic}, Theorem~\ref{thm:positivity}, and Corollary~\ref{cor:Q-Koszul} is satisfied in this case, by~\cite{kl}.

We will also consider a connected reductive group $\Gv$, a Borel subgroup $\Bv \subset \Gv$, and a maximal torus $\Tv \subset G$, such that the based root datum of $\Gv$ determined by $\Tv$ and $\Bv$ is dual to the based root datum of $G$ determined by $T$ and $B$.
As above we have a flag variety $\cBv:=\Gv/\Bv$, endowed with the Bruhat stratification. The strata are also parametrized by $W$ (since the Weyl groups of $(G,T)$ and $(\Gv,\Tv)$ can be canonically identified). We will use h\'a\v cek accents to denote objects attached to $\Gv$ rather than to $G$. For instance, $\dv_w(\E)$ is a standard object in $\Perv_{(\Bv)}(\cBv,\E)$, and $\cTv^\mix_w(\E)$ is a tilting object in $\Perv^\mix_{(\Bv)}(\cBv,\E)$.

Recall that by~\cite[Theorem~5.4]{modrap2} there exists an equivalence of triangulated categories
\[
\kappa : \Dmix_{(B)}(\cB,\E) \simto \Dmix_{(\Bv)}(\cBv,\E)
\]
which satisfies in particular $\kappa \circ \la n \ra \cong \la -n \ra [n] \circ \kappa$ and
\[
\kappa(\nmix_w) \cong \nv^\mix_{w^{-1}}, \quad \kappa(\cT^\mix_w) \cong \cEv^\mix_{w^{-1}}, \quad \kappa(\cE^\mix_w) \cong \cTv^\mix_{w^{-1}}.
\]
Below we will also use the Radon transform
\[
\Radon^\mix : \Dmix_{(B)}(\cB,\E) \simto \Dmix_{(B)}(\cB,\E).
\]
This equivalence of triangulated categories satisfies 
\[
\Radon^\mix(\nmix_w \la n \ra) \cong \dmix_{w w_0} \la n \ra, \qquad \Radon^\mix(\cT^\mix_w \la n \ra) \cong \cP^\mix_{w w_0} \la n \ra.
\]
(See~\cite[Proposition~4.11]{modrap2}.)
We also set
\[
\sigma := \kappa \circ (\Radon^\mix)^{-1} : \Dmix_{(B)}(\cB,\E) \simto \Dmix_{(\Bv)}(\cBv,\E).
\]
This functor has the property that
\[
\sigma(\dmix_w \la n \ra) \cong \nv^\mix_{w_0 w^{-1}} \la -n \ra [n]
\qquad\text{and}\qquad
\sigma(\cP^\mix_w \la n \ra) \cong \cEv^\mix_{w_0 w^{-1}} \la -n \ra [n].
\]

In~\cite[Proposition~5.5]{modrap2} we have also constructed a t-exact ``forgetful'' functor
\[
\mu : \Dmix_{(B)}(\cB,\E) \to \Db_{(B)}(\cB,\E)
\]
(where the right-hand side is endowed with the usual perverse t-structure)
and an isomorphism $\mu \circ \la 1 \ra$ such that for all $\cF,\cG \in \Dmix_{(B)}(\cB,\E)$ the morphism
\begin{equation}
\label{eqn:morphism-mu}
\bigoplus_{n \in \Z} \Hom(\cF,\cG \la n \ra) \to \Hom(\mu \cF, \mu \cG)
\end{equation}
induced by $\mu$ is an isomorphism, and such that
\begin{gather*}
\mu(\dmix_w) \cong \Delta_w, \quad \mu(\nmix_w) \cong \nabla_w, \quad \mu(\IC^\mix_w) \cong \IC_w, \\
\mu(\cT^\mix_w) \cong \cT_w, \quad \mu(\cE^\mix_w) \cong \cE_w.
\end{gather*}
(Here $\Delta_w$, $\nabla_w$, $\IC_w$, $\cT_w$ are the obvious ``non-mixed'' analogues of $\dmix_w$, $\nmix_w$, $\IC^\mix_w$, $\cT^\mix_w$, which are objects of the usual category $\Perv_{(B)}(\cB,\E)$ of Bruhat-constru\-ctible perverse sheaves on $\cB$.) There is also a functor $\muv : \Dmix_{(\Bv)}(\cBv,\E) \to \Db_{(\Bv)}(\cBv,\E)$ with similar properties.

\subsection{Main results}

The next two theorems are the main results of the paper.

\begin{thm}[Positivity]\label{thm:main1}
The following are equivalent:
\begin{enumerate}
\item The category $\Perv^\mix_{(B)}(\cB,\F)$ is positively graded.
\item For all $w \in W$, we have $\Delta^\circ_w(\O) \cong \IC^\mix_w(\O)$.
\item For all $w \in W$, the object $\cEv_w(\O) \in \Db_{(\Bv)}(\cBv,\O)$ is perverse.
\item For all $w \in W$, the object $\cEv_w(\F) \in \Db_{(\Bv)}(\cBv,\O)$ is perverse.
\end{enumerate}
\end{thm}
\begin{proof}
The equivalence of the first two statements follows from Theorem~\ref{thm:positivity}. The equivalence of the last two statements follows from the fact that the objects $\cEv_w(\O)$ have free stalks and costalks by definition.

By~\cite[Corollary~5.6]{modrap2}, the last statement is equivalent to the condition that $(\cT^\mix_v : \nmix_u\la n\ra ) = 0$ for all $n > 0$ and all $u,v \in W$. Using the equivalence $\Radon^\mix$, the latter is equivalent to requiring that $(\cP^\mix_v : \dmix_u \la n\ra) = 0$ for all $n > 0$ and all $u,v \in W$.  By Proposition~\ref{prop:perv-pos}, we conclude that the first and third statements are equivalent.
\end{proof}

\begin{thm}[$Q$-Koszulity]\label{thm:main2}
The following are equivalent:
\begin{enumerate}
\item The category $\Perv^\mix_{(B)}(\cB,\F)$ is metameric.\label{it:metameric}
\item The category $\Perv^\mix_{(\Bv)}(\cBv,\F)$ is standard $Q$-Koszul.\label{it:q-koszul}
\item For all $w \in W$, we have $\dv^\circ_w(\O) \cong \ICv^\mix_w(\O)$, and $\ICv^\mix_w(\O)$ has torsion-free stalks.\label{it:torsion-free}
\item For all $w \in W$, we have $\dv^\circ_w(\O) \cong \ICv^\mix_w(\O)$, and $\ICv_w(\O)$ has torsion-free stalks.\label{it:torsion-free-classical}
\end{enumerate}
\end{thm}
\begin{proof}
The equivalence
\eqref{it:q-koszul}${}\Longleftrightarrow{}$\eqref{it:torsion-free} 
follows from Corollary~\ref{cor:Q-Koszul}. The equivalence
\eqref{it:torsion-free}${}\Longleftrightarrow{}$\eqref{it:torsion-free-classical} follows from~\eqref{eqn:morphism-mu} (or rather its analogue for $\cBv$), using the fact that $\muv(\ICv^\mix_w(\O)) \cong \ICv_w(\O)$ and $\muv(\nv^\mix_v(\O)) \cong \nv_v(\O)$, and the observation that an object $M$ of the derived category of finitely generated $\O$-modules has free cohomology objects iff $\Hom^k(M,\O)$ is a free $\O$-module for all $k \in \Z$.

\eqref{it:metameric}${}\Longrightarrow{}$\eqref{it:torsion-free}.  Assuming that $\Perv^\mix_{(B)}(\cB,\F)$ is metameric, Theorem~\ref{thm:qher-wdelta} gives us a class of objects $\{ \wDelta^\mix_w \}_{w \in W}$ in $\Perv^\mix_{(B)}(\cB,\F)$.  Form the short exact sequence $K_w \hookrightarrow \wDelta^\mix_w \twoheadrightarrow \dmix_w(\F)$.  Recall that $K_w$ has a filtration by various $\dmix_u\la n\ra$ with $n < 0$.  Therefore, $\sigma(K_w)$ is an iterated extension of various $\nv^\mix_u \la -n \ra [n]$ with $n < 0$.  In particular, $\sigma(K_w) \in \Dmix_{(\Bv)}(\cBv,\F)_{\bgeq 1}$.

On the other hand, the $\wDelta^\mix_w$ have the property that $\Ext^k(\wDelta^\mix_w, \dmix_u \la n\ra) = 0$ for all $k$ and all $n < 0$.  Applying $\sigma$, we obtain that
\[
\Ext^k(\sigma(\wDelta^\mix_w), \nv^\mix_{w_0 u^{-1}} \la -n \ra [n]) = 0 \qquad\text{for all $k$ and all $n < 0$.}
\]
This implies that $\sigma(\wDelta^\mix_w) \in \Dmix_{(\Bv)}(\cBv,\F)_{\bleq 0}$ (see Remark~\ref{rmk:baric-morphisms}).  Thus, the following two distinguished triangles must be isomorphic:
\[
\sigma(\wDelta^\mix_w) \to \sigma(\dmix_w) \to \sigma(K_w[1]) \xrightarrow{[1]}, \quad
\nv^\circ_{w_0 w^{-1}} \to \nv^\mix_{w_0 w^{-1}} \to \beta_{\bgeq 1}\nv^\mix_{w_0 w^{-1}} \xrightarrow{[1]}
\]
(see Lemma~\ref{lem:baric-basic}\eqref{it:baric-tri});
in particular we obtain an isomorphism 
\[
\sigma(\wDelta^\mix_w) \cong \nv^\circ_{w_0 w^{-1}}.
\]
Now, $\wDelta^\mix_w$ is an iterated extension of various $\dmix_u\la n\ra$, so $\nv^\circ_{w_0 w^{-1}}$ is an iterated extension of various $\nv^\mix_u\{ n\}$.  In particular, the costalks of $\nv^\circ_{w_0 w^{-1}}$ are extensions of the costalks of the $\nv^\mix_u\{ n\}$.  The latter are pure of weight~$0$, so the same holds for $\nv^\circ_{w_0 w^{-1}}$.  By Verdier duality, the stalks of the objects $\dv^\circ_{w_0 w^{-1}}$ are pure of weight~$0$.  By Lemma~\ref{lem:stalks-dcirc-ic}, we find that condition~\eqref{it:torsion-free} holds.

\eqref{it:torsion-free}{}${}\Longrightarrow{}$\eqref{it:metameric}. When~\eqref{it:torsion-free} holds, by Theorem~\ref{thm:main1} the category $\Perv^\mix_{(\Bv)}(\cBv,\F)$ is positively graded. First, let us prove that the category $\Perv^\mix_{(B)}(\cB,\F)$ also is positively graded. 
By Lemma~\ref{lem:perv-circ-qher}, $\Perv^\circ_{(\Bv)}(\cBv,\F)$ is a quasihereditary category (see Remark~\ref{rmk:dcirc-perverse}).  We claim that the indecomposable tilting objects in this category are the parity sheaves $\cEv_w^\mix$. Indeed, let $\cTv^\circ_w$ be the unique indecomposable tilting object of $\Perv^\circ_{(\Bv)}(\cBv,\F)$ whose support is $\overline{\cBv_w}$.  Then $\cTv^\circ_w$ has a filtration by various $\dv^\circ_u$, whose stalks are pure of weight $0$ by Lemma~\ref{lem:stalks-dcirc-ic}.  Therefore, the stalks of $\cTv^\circ_w$ are also pure of weight~$0$. By Lemma~\ref{lem:wt-stalks}, it follows that $\cTv^\circ_w$ has weights${}\le 0$. Since $\cTv^\circ_w$ is Verdier-self-dual, it also has weights${}\ge 0$, so it must be pure of weight $0$, and hence a direct sum of various $\cEv^\mix_u\{m\}$ by the remarks in~\S\ref{ss:weights}. By indecomposability and support considerations, we even obtain that $\cTv_w^\circ \cong \cEv^\mix_w \{m\}$ for some $m \in \Z$. Considering the restriction to $\cBv_w$, we obtain that $m=0$, i.e.~that $\cTv_w^\circ \cong \cEv^\mix_w$, as claimed. Now the objects $\dv^\circ_u$ are perverse sheaves (see Remark~\ref{rmk:dcirc-perverse}), so the objects $\cTv_w^\circ$ are also perverse; we deduce that $\cEv^\mix_w$ is a perverse sheaf for any $w \in W$. Using Theorem~\ref{thm:main1} again, this finishes the proof of the fact that $\Perv^\mix_{(B)}(\cB,\F)$ is positively graded.

To conclude, we will essentially reverse the argument used in the proof of the implication \eqref{it:metameric}${}\Longrightarrow{}$\eqref{it:torsion-free}.  Let us define $\wDelta^\mix_w \in \Dmix_{(B)}(\cB,\F)$ to be $\sigma^{-1}(\nv^\circ_{w_0 w^{-1}})$.  Since~\eqref{it:torsion-free} holds, using Lemma~\ref{lem:stalks-dcirc-ic} and Verdier duality, we know that the costalks of $\nv^\circ_w$ are pure of weight $0$ for all $w \in W$.  In other words, for any $u \in W$, the object $\iv_{u*}\iv_u^! \nv^\circ_w$ is a direct sum of various $\nv^\mix_u\{n\}$ with $n \in \Z$. In fact, since $\nv^\circ_w$ is perverse (see Remark~\ref{rmk:dcirc-perverse}) we must have
$n \le 0$. We even have $n < 0$ unless $u = w$, and in that case, we have $\nv^\mix_w=
\iv_{w*}\iv_w^!\nv^\circ_w$. (The first claim follows from the following computation for $u<w$: $\Hom(\uuF, \iv_u^! \nv^\circ_w) \cong \Hom(\dv^\mix_u, \nv^\circ_w) \cong \Hom(\dv^\circ_u,\nv^\circ_w)=0$, where the second isomorphism follows from~\eqref{eqn:pos-higher-ext}. The second claim is obvious from the construction of $\nv^\circ_w$ in Proposition~\ref{prop:parshall-scott}.)
A routine recollement argument shows that $\nv^\circ_w$ is an iterated extension of the various $\iv_{u*}\iv_u^!\nv^\circ_w$, and hence of $\nv^\mix_w$ together with various $\nv^\mix_u\{n\}$ with $n < 0$ and $u<w$.  Applying $\sigma^{-1}$ to this description, we find that $\wDelta^\mix_w$ is an iterated extension of $\dmix_w$ and various $\dmix_u\la n\ra$ with $n < 0$ and $u<w$.  In particular, $\wDelta^\mix_w$ is a perverse sheaf with a standard filtration.  

Next, we claim that $\Ext^k(\wDelta^\mix_w, \IC^\mix_v\la n\ra) = 0$ for all $k \ge 1$ if $n \ne 0$, or else if $n = 0$ and $v \le w$. For $n>0$ this follows from~\eqref{eqn:pos-higher-ext}, and one can easily check using induction on $v$ that the conditions for $n \leq 0$ are equivalent to 
\[
\Ext^k(\wDelta^\mix_w, \dmix_v\la n\ra) = 0 \qquad\text{for all $k \ge 1$ if $n < 0$, or else if $n = 0$ and $v \le w$.}
\]
Applying $\sigma$, this is equivalent to a similar vanishing claim about
\[
\Hom(\nv^\circ_{w_0 w^{-1}}, \nv^\mix_{w_0 v^{-1}}\la -n\ra[k+n]).
\]
If $n<0$, this claim follows from~\eqref{eqn:pos-higher-ext}. If $n=0$, it holds for reasons of support.

Referring to Theorem~\ref{thm:qher-wdelta}, we see that we have already shown that the objects $\wDelta^\mix_w$ enjoy properties~\eqref{it:qher-ext0}, \eqref{it:qher-ext-neg}, and~\eqref{it:qher-filt}.  We will now show that they satisfy property~\eqref{it:qher-head} as well.  The $\Ext^1$-case of the vanishing proved above shows that $\wDelta^\mix_w$ is projective as an object of the Serre subcategory of $\Perv^\mix_{(B)}(\cB,\F)$ generated by all $\IC^\mix_v\la n\ra$ with $n < 0$, together with the $\IC^\mix_v$ with $v \le w$. It is indecomposable because $\nv^\circ_w$ is, so it is the projective cover of some simple object.  Its unique simple quotient must be the head of one of the standard objects in its standard filtration.  By weight filtration considerations, that unique simple quotient must be $\IC^\mix_w$.

We have shown that the objects $\wDelta^\mix_w$ satisfy the properties listed in Theorem~\ref{thm:qher-wdelta}. It is clear that the objects $\wnabla^\mix_w:=\D(\wDelta^\mix_w)$ will satisfy the dual conditions, so that Theorem~\ref{thm:qher-wdelta} implies that $\Perv^\mix_{(B)}(\cB,\F)$ is metameric.
\end{proof}

\begin{rmk}
When the conditions of Theorem~\ref{thm:main2} are satisfied, one can complete the description~\cite[Figure~1]{modrap2} of the behavior of the various special objects under the equivalence $\kappa$, as shown in Figure~\ref{fig}. (Here the isomorphism on the third line follows from Lemma~\ref{lem:perv-circ-qher}, and question marks indicate objects for which we don't have an explicit description.)
\begin{figure}
\[
\xymatrix@R=0.1cm@C=1.5cm{
\Dmix_{(B)}(\cB,\F) \ar@/^4ex/[rr]^-{\sigma} & \Dmix_{(B)}(\cB,\F)  \ar[l]^-{\Radon^\mix}_-{\sim} \ar[r]_-{\kappa}^-{\sim} & \Dmix_{(B)}(\cB,\F) \\
? & \dmix_w \ar@{|->}[l] \ar@{|->}[r] & \dv^\mix_{w^{-1}} \\
\dmix_{w w_0} & \nmix_w \ar@{|->}[l] \ar@{|->}[r] & \nv^\mix_{w^{-1}} \\
? & \cE_w^\mix \cong \cT^\circ_w \ar@{|->}[l] \ar@{|->}[r] & \cTv^\mix_{w^{-1}} \\
\cP^\mix_{w w_0} & \cT^\mix_w \ar@{|->}[l] \ar@{|->}[r] & \cEv_{w^{-1}}^\mix \cong \cTv^\circ_{w^{-1}} \\
\wDelta^\mix_{w w_0} & ? \ar@{|->}[l] \ar@{|->}[r] & \nv^\circ_{w^{-1}}
}
\]
\caption{Behavior of various objects under the equivalence $\sigma$.}
\label{fig}
\end{figure}
\end{rmk}

\begin{rmk}
Suppose that the conditions in Theorem~\ref{thm:main2} hold.  By Lemma~\ref{lem:perv-circ-qher}, the $\cEv^\mix_w(\F)$ are precisely the indecomposable tilting objects in $\Perv^\circ_{(\Bv)}(\cBv,\F)$.  Since the equivalence $\sigma^{-1}: \Dmix_{(\Bv)}(\cBv,\F) \to \Dmix_{(B)}(\cB,\F)$ takes these to projective objects in $\Perv^\mix_{(B)}(\cB,\F)$, the category $\Perv^\mix_{(B)}(\cB,\F)$ is the ``$T$-Koszul dual'' to $\Perv^\mix_{(\Bv)}(\cBv,\F)$ in the sense of Madsen~\cite{madsen}. (See Remark~\ref{rmk:t-koszul}.)
\end{rmk}

\subsection{Koszulity}

We conclude this paper with a proof of the converse to Corollary~\ref{cor:parity-koszul}, in the case of flag varieties.

\begin{thm}
\label{thm:koszulity}
The following are equivalent:
\begin{enumerate}
\item For all $w \in W$ we have $\cE_w(\F) \cong \IC_w(\F)$.\label{it:IC-parity}
\item The category $\Perv^\mix_{(B)}(\cB,\F)$ is Koszul.\label{it:koszul}
\item The category $\Perv^\mix_{(B)}(\cB,\F)$ is positively graded, and $\Perv^\mix_{(B)}(\cB,\F)^\circ$ is a semisimple category.\label{it:koszul-weak}
\end{enumerate}
Moreover, these statements hold if and only if their analogues for $\cBv$ hold.
\end{thm}

\begin{proof}
In this proof, we will write \eqref{it:IC-parity}$^\vee$ to refer to the analogue of statement~\eqref{it:IC-parity} for $\cBv$, and likewise for the other assertions in the theorem.

The implications
\eqref{it:IC-parity}${}\Longrightarrow{}$\eqref{it:koszul} and \eqref{it:IC-parity}$^\vee${}${}\Longrightarrow{}$\eqref{it:koszul}$^\vee$
follow from Corollary~\ref{cor:parity-koszul}. 

The implications \eqref{it:koszul}${}\Longrightarrow{}$\eqref{it:koszul-weak} and \eqref{it:koszul}$^\vee${}${}\Longrightarrow{}$\eqref{it:koszul-weak}$^\vee$ are obvious.

\eqref{it:koszul-weak}${}\Longrightarrow{}$\eqref{it:IC-parity}$^\vee$. Since $\Perv^\mix_{(B)}(\cB,\F)$ is positively graded, by Theorem~\ref{thm:main1}, $\cEv_w(\F)$ is perverse. Now, 
the fact that $\Perv^\mix_{(B)}(\cB,\F)^\circ$ is semisimple implies that the ring
\[
\Hom_{\Perv^\mix_{(B)}(\cB,\F)} \left( \bigoplus_{v \in W} \cP^\mix_v(\F), \bigoplus_{v \in W} \cP^\mix_v(\F) \right) 
\]
is isomorphic to $\bigoplus_v \bk$ (where $1$ in the copy of $\bk$ parametrized by $v$ corresponds to the identity morphism of $\cP^\mix_v(\F)$). Using equivalence $\sigma$, we deduce a similar claim for the objects $\cEv^\mix_v(\F)$, $v \in W$. It follows that
\begin{equation}
\label{eqn:morphisms-parity}
\Hom_{\Db_{(B)}(\cB,\F)} \bigl( \cEv_v(\F), \cEv_w(\F) \bigr) = 0 \qquad \text{unless $v=w$.}
\end{equation}
Now assume that there exists $w \in W$ such that the perverse sheaf $\cEv_w(\F)$ is not simple, and choose $w \in W$ minimal (for the Bruhat order) with this property. Since $\cEv_w(\F)$ is supported on the closure of $\cBv_w$, and since its restriction to $\cBv_w$ is $\uuF$, either the top or the socle of $\cEv_w(\F)$ contains a simple object $\ICv_v(\F)$ with $v < w$. Then there exists either a non zero morphism $\cEv_w(\F) \to \ICv_v(\F)$, or a nonzero morphism $\ICv_v(\F) \to \cEv_w(\F)$. Since $\ICv_v(\F) \cong \cEv_v(\F)$ by minimality, this contradicts~\eqref{eqn:morphisms-parity} and finishes the proof of the implication.

By symmetry we also obtain the implication \eqref{it:koszul-weak}$^\vee${}${}\Longrightarrow{}$\eqref{it:IC-parity}, which finishes the proof.
\end{proof}

\end{document}